\documentclass[12pt]{article}

\usepackage[margin=1in]{geometry}  
\usepackage{graphicx}              
\usepackage{amsmath}               
\usepackage{amsfonts}              
\usepackage{amsthm}                
\usepackage{color}
\usepackage{tabulary}
\usepackage{caption}
\usepackage{subcaption}
\usepackage{amssymb}
\usepackage{subfig}

\newenvironment{theorem}[2][Theorem]{\begin{trivlist}
\item[\hskip \labelsep {\bfseries #1}\hskip \labelsep {\bfseries #2}]}{\end{trivlist}}

\newenvironment{corollary}[2][Corollary]{\begin{trivlist}
\item[\hskip \labelsep {\bfseries #1}\hskip \labelsep {\bfseries #2}]}{\end{trivlist}}

\newtheorem{thm}{Theorem}[section]
\newtheorem{lem}[thm]{Lemma}
\newtheorem{prop}[thm]{Proposition}
\newtheorem{cor}[thm]{Corollary}
\newtheorem{conj}[thm]{Conjecture}
\newtheorem{defn}[thm]{Definition}
\newtheorem{rmk}[thm]{Remark}

\newtheorem{ques}[thm]{Question}

\DeclareMathOperator{\PR}{PSL_2(\RR)}

\DeclareMathOperator{\Aut}{Aut}
\DeclareMathOperator{\Homeo}{Homeo}
\DeclareMathOperator{\Homeop}{Homeo_+}

\DeclareMathOperator{\CO}{CO}
\DeclareMathOperator{\COL}{COL}

\DeclareMathOperator{\Stab}{Stab}
\DeclareMathOperator{\Fix}{Fix}

\newcommand{\NN}{\mathbb{N}}      %
\newcommand{\ZZ}{\mathbb{Z}}      
\newcommand{\RR}{\mathbb{R}}      
\newcommand{\DD}{\mathbb{D}}      
\newcommand{\HH}{\mathbb{H}}      

\date{}

\begin{document}

\title{Fuchsian Groups, Circularly Ordered Groups, and Dense Invariant Laminations on the Circle}

\author{Hyungryul Baik}

\maketitle

\abstract{ We propose a program to study groups acting faithfully on $S^1$ in terms of number of pairwise transverse dense invariant laminations. 
We give some examples of groups which admit a small number of invariant laminations as an introduction to such groups. Main focus of the present paper is to characterize Fuchsian groups in this scheme. We prove a group acting on $S^1$ is conjugate to a Fuchsian group if and only if it admits 
three very-full laminations with a variation of the transversality condition. Some partial results toward a similar characterization of hyperbolic 3-manifold groups which fiber over the circle have been obtained. This work was motivated by the universal circle theory for tautly foliated 3-manifolds developed by Thurston and Calegari-Dunfield. } 


\section{Introduction} 
\label{sec:introduction}
We say a group is $\CO$ if it is circularly orderable. See \cite{Calegari04} for general background for circular ordering of groups. It is well known that a group is $\CO$
if and only if it acts faithfully on $S^1$. In this paper,  we only talk about circularly ordered groups. More precisely, a group $G$ comes with an injective homomorphism 
from $G$ to $\Homeop(S^1)$ where $\Homeop(S^1)$ is the group of all orientation-preserving homeomorphisms of $S^1$. 
Abusing the notation, we identify $G$ with its image under this representation and regard it as a subgroup of $\Homeop(S^1)$, ie., 
we consider the subgroups of $\Homeop(S^1)$. There is a reason why we emphasize this: the properties we will define may depend on the circular order on a group. 
So, if we just talk about abstract group which is circularly orderable without specifying the actual circular order, there is a possible ambiguity. 
Since we care only about the topological dynamics, the groups are considered up to topological conjugacy, i.e., conjugacy by elements of $\Homeop(S^1)$. 

Circularly orderable groups arise naturally in low-dimensional topology. Thurston showed that for a 3-manifold $M$ admitting a taut foliation, $\pi_1(M)$ admits a faithful action on the circle (which is now called a universal circle) in his unfinished manuscript (see \cite{Thurston97}). In \cite{CalDun03}, Calegari and Dunfield completed the construction and generalized this to 3-manifolds admitting essential laminations with solid torus guts. Universal circles from taut foliations come with a pair of transverse dense invariant laminations. 
This provides a motivation to study those groups acting on $S^1$ with some invariant laminations. We suggest a new classification of the subgroups of $\Homeop(S^1)$ in terms of the number of dense invariant laminations they admit. 
In this paper, we mainly focus on the case of groups acting faithfully on $S^1$ with two or three different very-full invariant laminations. We also give motivation for this classification by demonstrating interesting examples and questions. 

By a Fuchsian group, we mean a torsion-free discrete subgroup of $\PR$ (up to conjugacy by an element of $\Homeop(S^1)$). Recall that $\PR$ is naturally identified with the group of orientation-preserving isometries of the hyperbolic plane $\HH^2$. For a collection $\mathcal{C}$ of $G$-invariant laminations, being \emph{pants-like} means that a pair of leaves from two different laminations in $\mathcal{C}$ shares a common endpoint if and only if the shared endpoint is the fixed point of a parabolic element of $G$. For other terminologies, see Section 2. 
\begin{theorem}{[Main Theorem]}
Let $G$ be a torsion-free discrete subgroup of $\Homeop(S^1)$. Then $G$ is a Fuchsian group such that $\HH^2/G$ is not the thrice-punctured sphere if and only if 
$G$ admits a pants-like collection $\{\Lambda_1, \Lambda_2, \Lambda_3\}$ of three very-full $G$-invariant laminations. 
\end{theorem} 

As we pointed out earlier, saying a group $G$ is Fuchsian means $G$ is conjugate to a group $G' \subset \PR$  by an element of $\Homeop(S^1)$, and $\HH^2/G$ in the statement of the theorem should be understood as $\HH^2/G'$. The theorem provides an alternative characterization of Fuchsian groups in terms of invariant laminations. Note that we do not assume that $G$ is finitely generated. The following is an immediate corollary of Main Theorem. 

\begin{corollary}{}
Let $G$ be a torsion-free discrete subgroup of $\Homeop(S^1)$. Then $G$ is a Fuchsian group such that $\HH^2/G$ has no cusps if and only if $G$ admits a collection $\{\Lambda_1, \Lambda_2, \Lambda_3\}$ of three very-full $G$-invariant laminations such that no leaf of $\Lambda_i$ has a common endpoint of a leaf of $\Lambda_j$ for $i \neq j$. 
\end{corollary}

In Section 3, we present some explicit examples of groups acting on the circle with a specified number of dense invariant laminations. The most interesting case is when a group has exactly two dense invariant laminations. A class of examples will be constructed by considering pseudo-Anosov homeomorphisms of hyperbolic surfaces. One should note that those examples are not Fuchsian groups. In some sense, the main theorem shows that there are clear differences between having two invariant laminations and having three invariant laminations as long as the structure of the invariant laminations is restricted enough.  

Nevertheless, groups admitting a pant-like collection of two very-full laminations are already interesting. We study those groups in Section 8 and the following theorem is a summary of the results. 

\begin{theorem}{} 
Let  $G$ be a torsion-free discrete subgroup of $\Homeop(S^1)$. Suppose $G$ admits a pants-like collection of two very-full laminations $\{\Lambda_1, \Lambda_2\}$. Then an element of $G$ either behaves like a parabolic or hyperbolic isometry of $\HH^2$ or has even number of fixed points alternating between attracting and repelling. In the latter case, one lamination contains the boundary leaves of the convex hull of the attracting fixed points and the other lamination contains the boundary leaves of the convex hull of the repelling fixed points. If we further assume that $G$ has no element with single fixed point, then $G$ acts faithfully on $S^2$ by orientation-preserving homeomorphisms and each element of $G$ has two fixed points on $S^2$. 

\end{theorem} 

\noindent \textbf{Acknowledgment} The author started studying the subject of this paper as a student of Bill Thurston who recently passed away. We give tremendous amount of thanks to Bill for the inspiration and support. Dylan Thurston has helped to develop ideas and pointed out whenever the author had some fallacy in the argument. More importantly, Dylan has given a lot of encouragement so that the author has not given up writing up this paper. We would like to thank Juan Alonso, Danny Calegari, John Hubbard and Michael Kapovich for fruitful conversations. Most figures in this paper are due to Eugene Oh and I greatly appreciate. We also thank to the department of Mathematics at UC Davis for the hospitality in the academic year 2012-13. The research was partially supported by Bill Thurston's NSF grant DMS-0910516.


\section{Definitions and Set-up} 
\label{sec:setup}
In the present paper, a group is always assumed to be countable. A faithful (orientation-preserving) action of a group $G$ on the circle is an injective homomorphism $\rho: G \to \Homeop(S^1)$. Once we fix the action, then we often identify $G$ with its image under $\rho$. For the general background on the group actions on the circle, we suggest reading \cite{Ghys01}. 
 \begin{figure}	
	\centering
	\includegraphics[scale=0.5]{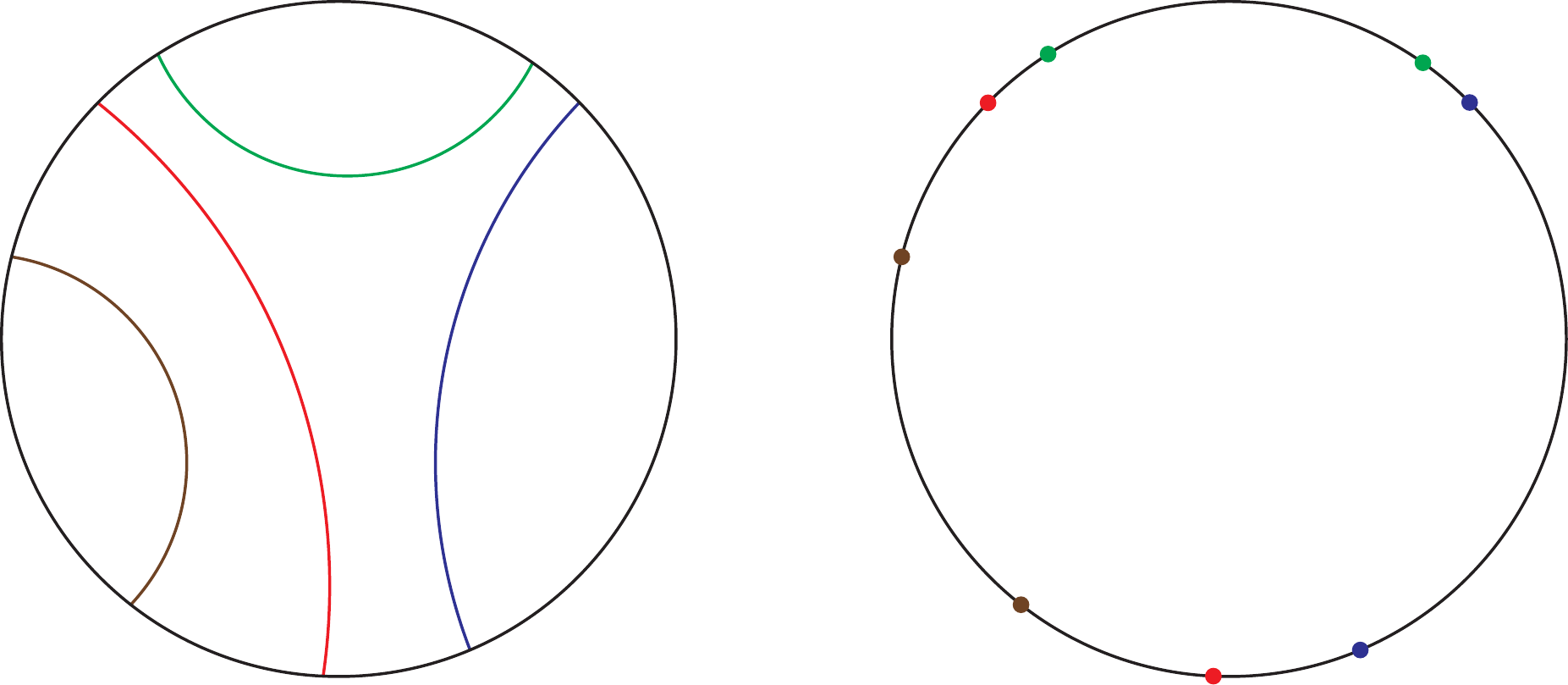}
	\caption{ On the left, a geodesic lamination on $\HH^2$ with four leaves is drawn. On the right, one can see the corresponding lamination of the circle after removing the geodesics and leaving only the endpoints.}
	\label{fig:laminationonthecircles}
\end{figure} 

 The ideal boundary of the hyperbolic plane $\HH^2$ is topologically a circle. A geodesic lamination of $\HH^2$ is a disjoint union of geodesics which is a closed subset of $\HH^2$. If one forgets about actual geodesics of a geodesic lamination of $\HH^2$ and consider only the endpoints of geodesics on the ideal boundary, one gets a set of pairs of points of the circle.  A lamination on the circle is defined as a set of pairs of points of the circle to capture this endpoint data of a geodesic lamination of $\HH^2$ (see Figure \ref{fig:laminationonthecircles}).

Two pairs $(p_1, p_2), (q_1, q_2)$ of unordered two distinct points of $S^1$ are called \emph{linked} if the chord joining $p_1$ to $p_2$ crosses the chord joining $q_1$ to $q_2$ in the interior of the disk bounded by $S^1$. The space of all unordered pairs of two distinct points of $S^1$ is $((S^1 \times S^1) \setminus \Delta)/ (x,y) \sim (y,x) $ where $\Delta =\{ (x,x) \in S^1 \times S^1 \}$. This is homeomorphic to an open M\"{o}bius band and we will denote this space $\mathcal{M}$. A group action on $S^1$ induces an action on $\mathcal{M}$ in the obvious way; this action is not minimal in our examples, since otherwise there could not be any invariant laminations. 

A \emph{lamination} on $S^1$ is a set of unordered and unlinked pairs of two distinct points of $S^1$ which is a closed subset of $\mathcal{M}$. The elements of a lamination (which are pairs of points of $S^1$) are called \emph{leaves} of the lamination. If a leaf is the pair of points $(p, q)$, then the points $p, q$ are called \emph{ends} or \emph{endpoints} of the leaf. For a lamination $\Lambda$, let $E_\Lambda$ or $E(\Lambda)$ denote the set of all ends of leaves of $\Lambda$. 
We also use $\overline{\mathcal{M}}$ to denote the closed M\"{o}bius band and the points on $\partial \mathcal{M} := \overline{\mathcal{M}} \setminus \mathcal{M}$ are called \emph{degenerate} leaves which are single points of $S^1$. 

Alternatively, one can identify the circle with the ideal boundary of $\HH^2$ and consider only the ends of leaves of some geodesic lamination of $\HH^2$. Every lamination on $S^1$ is of this form. Even though the group action on $S^1$ does not extend to the interior of the disk, it is usually better to picture a lamination of $S^1$ as a geodesic lamination of $\HH^2$. Consider a connected component of the complement of the lamination in the open disk. Its closure in the closed disk is called a \emph{gap} or a \emph{complementary region} of the lamination. In other words, a gap of a lamination $\Lambda$ of $S^1$ is the metric completion of a connected component of the complement of the corresponding lamination in $\DD$ with respect to the path metric. We will use $\DD$ to denote the open disk bounded by $S^1$ where the groups we consider act. The disk $\DD$ will be freely identified with the Poincar\'{e} disk model of $\HH^2$ often without mentioning it if there is no confusion. 

Once a group $G$ acts on $S^1$ by homeomorphisms, there is a diagonal action on $\mathcal{M}$. A lamination $\Lambda$ of $S^1$ is said to be $G$-invariant if it is an invariant subset of $\mathcal{M}$ under this induced action of $G$. 

We give names to some properties of laminations. 
\begin{defn} Let $G$ be a group acting on $S^1$ faithfully. A $G$-invariant lamination $\Lambda$ is called
\begin{itemize} 
\item \emph{dense} if the endpoints of the leaves of $\Lambda$ form a dense subset of $S^1$, 
\item \emph{very-full} if all the gaps of $\Lambda$ are finite-sided ideal polygons, 
\item \emph{minimal} if the orbit closure of any leaf of $\Lambda$ is the whole $\Lambda$, 
\item \emph{totally disconnected} if no open subset of $\DD$ is foliated by $\Lambda$, 
\item \emph{solenoidal} if it is totally disconnected and has no isolated leaves, 
\item \emph{boundary-full} if the closure of the lamination in $\overline{\mathcal{M}}$ contains the entire $\partial \mathcal{M}$. 
\end{itemize} 
\end{defn} 
In fact, all properties above except the minimality are independent on the group action. Hence we use those notions for laminations on $S^1$ even when we do not have a group action in consideration. In this paper, very-full laminations are of a particular interest. See Figure \ref{fig:Farey_Graph} for an example \footnote{This figure is borrowed from Lars Madsen at Aarhus University.}. 

\begin{figure}	
	\centering
	\includegraphics[scale=0.3]{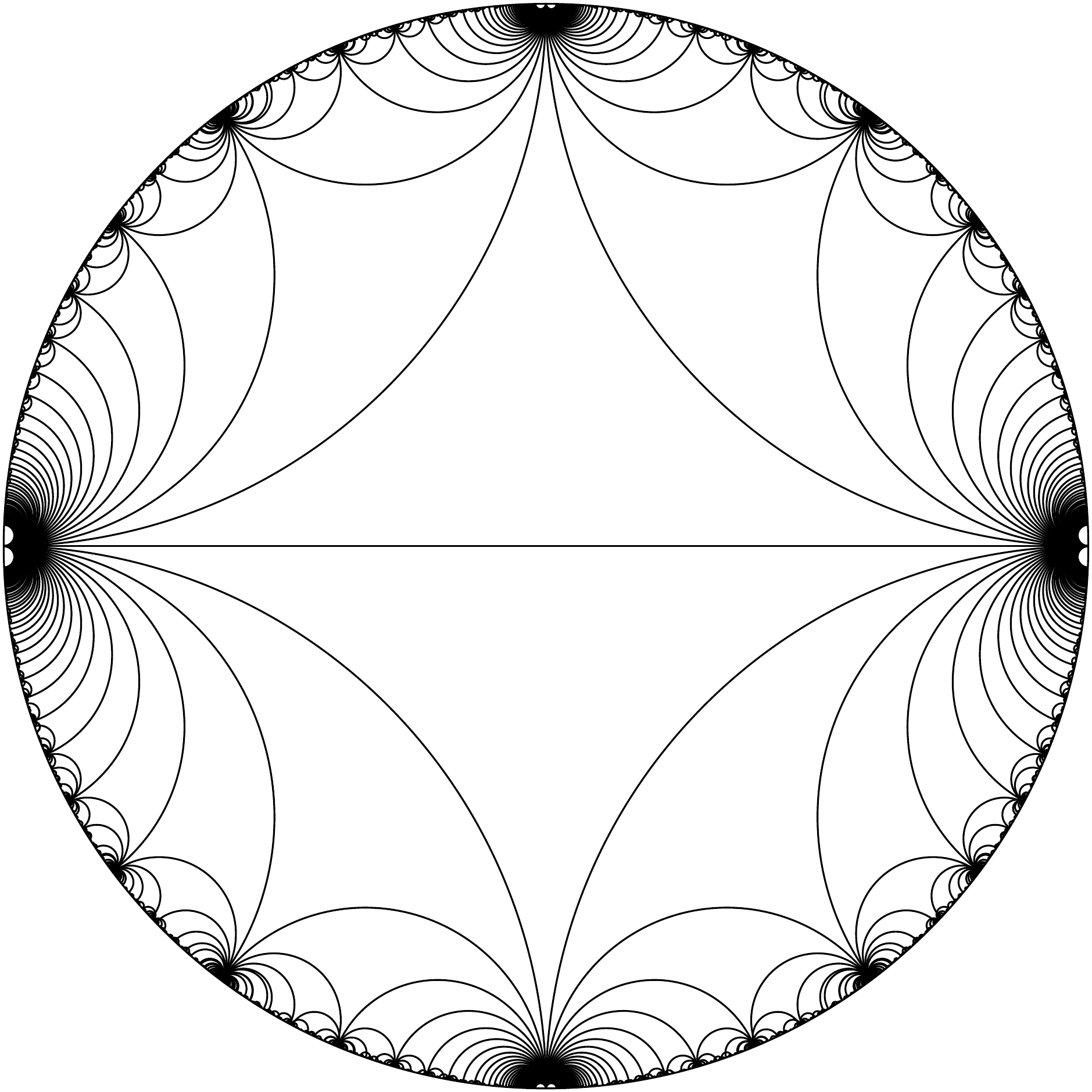}
	\caption{ Farey diagram is a rather famous example of very-full laminations.}  
	\label{fig:Farey_Graph}
\end{figure} 

A continuous map $f$ from $S^1$ to itself of degree 1 is called a \emph{monotone} map if the pre-image of each point in the range under $f$ is connected. Let $\rho_1: G \to \Homeop(S^1)$ and $\rho_2: G \to \Homeop(S^1)$ be faithful group actions on $S^1$. We say that $\rho_1$ is \emph{semi-conjugate} to $\rho_2$ if there exists a monotone map $f: S^1 \to S^1$ such that $f \circ \rho_1(g) = \rho_2(g) \circ f $ for all $g\in G$. If $f$ could be taken to be a homeomorphism, then $\rho_1$ is said to be \emph{conjugate} to $\rho_2$. Note that a semi-conjugacy (or rather a monotone map) gives a map from $\mathcal{M}$ to $\overline{\mathcal{M}}$. For general background on the laminations on $S^1$ and monotone maps, we highly recommend the Chapter 2 of \cite{Calebook}.

A group $G$ is said to act minimally on $S^1$ if all orbits are dense. One immediate consequence of an action being minimal is that the only non-empty closed $G$-invariant subset of $S^1$ is the entire $S^1$.Note that the minimality of an action of a group $G$ is not equivalent to the minimality of a $G$-invariant lamination. 

Some elements of $\Homeop(S^1)$ are particularly interesting to us. 
\begin{defn} For $g \in \Homeop(S^1)$, let $\Fix_g$ be the fixed-point set $\{ x \in S^1 : g(x) = x \}$. An element $g$ of $\Homeop(S^1)$ is said to be; 
\begin{itemize}
 \item an \emph{elliptic element} if $| Fix_g | = 0$,
 \item a \emph{parabolic element} if $| Fix_g | = 1$, 
 \item a \emph{hyperbolic element} if $| Fix_g | = 2$ and one fixed point is attracting and the other one is repelling. 
 \item a \emph{pseudo-Anosov-like element} if there exists $m > 0$ such that $|Fix_{g^m}| = 2n$ for some $n > 1$ and the elements of $\Fix_{g^m}$ alternate between attracting and repelling fixed points along $S^1$.
\end{itemize} 
\end{defn} 

Once a group $G \subset \Homeop(S^1)$ is given, a point $p$ on $S^1$ is called a \emph{cusp point} if $p$ is the fixed point of a parabolic element of $G$. 

Once we consider more than one lamination at the same time, we need some more definitions. 

\begin{defn}Two laminations $\Lambda_1, \Lambda_2$ of $S^1$ are \emph{transverse} if they have no leaf in common, ie., $\Lambda_1 \cap \Lambda_2 = \emptyset$ as subsets of $\mathcal{M}$. They are said to be \emph{strongly transverse} if no leaf of $\Lambda_1$ shares any endpoints with a leaf of $\Lambda_2$, ie., $E(\Lambda_1) \cap E(\Lambda_2) = \emptyset$. \end{defn}

For a collection of very-full laminations each of which is invariant under some group $G$, one can define a notion which lies between pairwise transversality and pairwise strong-transversality. The motivation of the following definition will be explained later

\begin{defn}  Let $G$ be a group acting on $S^1$ faithfully and let $\mathcal{C}= \{\Lambda_\alpha\}_{\alpha \in J}$ be a collection of $G$-invariant very-full laminations, where $J$ is an index set. Then $\mathcal{C}$ is called \emph{pants-like} if the laminations in $\mathcal{C}$ are pairwise transverse, and each point $p \in S^1$ is either fixed by a parabolic element of $G$ or an endpoint of a leaf of at most one lamination $\Lambda_\alpha$. In other words, for $\alpha \neq \beta \in J$, $E(\Lambda_\alpha) \cap E(\Lambda_\beta) = \{ \mbox{cusp points of } G \}$.  
\end{defn}

For a group $G \subset \Homeop(S^1)$, we say $G$ is $\COL_n$ for some $n \in \NN$ if it admits $n$ pairwise transverse dense invariant laminations. We use $\COL_\infty$ to denote the groups which admit an infinite collection of transverse dense invariant laminations.

By definition, we have the inclusions $$\COL_1 \supset \COL_2 \supset \COL_3 \supset \cdots. $$
We say a group is strictly $\COL_n$ if it is $\COL_n$ but not $\COL_{n+1}$. A $\COL_n$ group $G$ is said to be pants-like $\COL_n$ if the collection of $n$ pairwise transverse dense $G$-invariant laminations could be chosen to be pants-like. 

For an abstract group $G$ and an injective homomorphism $\rho: G \to \Homeop(S^1)$, $\rho$ is called $\COL_n$-representation if $\rho(G)$ is a $\COL_n$ group. One aim is to deduce interesting properties of a $\COL_n$ group from the dynamical and geometric data of its invariant laminations. 

 

We will consider the following natural questions.  

\begin{ques} 
\label{ques:coldiff}
Is the set of $\COL_i$ groups strictly bigger than the set of $\COL_{i+1}$ groups for any $i$? Can one characterize those groups in an interesting way? 
\end{ques} 
\begin{ques}
\label{ques:colinf}
Is $\COL_n$ nonempty for all $n$? 
\end{ques} 

We will get partial answers to Question \ref{ques:coldiff} and provide an affirmative answer to Question \ref{ques:colinf}. 
Our main result of the present paper is to show that pants-like $\COL_3$ groups are Fuchsian.


\section{Groups with Specified Number of Invariant Laminations} 
\label{sec:col12inf}
\subsection{Strictly $\COL_1$ groups}

In this section, we construct an example of a strictly $\COL_1$ group.
Let $R$ be a rigid rotation by an irrational angle and pick a point $p \in S^1$. Let $\mathcal{O}_p$ be the orbit of $p$ under the forward and backward iterates of $R$. Then 
it is a countable dense subset of $S^1$. Let $p_i = R^i(p)$, where $R^i$ is the $i$th iterate of $R$. We blow up all points in $\mathcal{O}_p$ and replace them by intervals. More precisely, 
replace $p_j$ by an interval of length $1/2^{|j|}$, and call this interval $I_j$.  Since the sum of the lengths of the $I_j$ is finite, we get again a circle. The action of $R$ on the new circle is the same as 
in the original circle in the complements of the $I_j$ and $R(I_j) = I_{j+1}$ is defined as a unique affine homeomorphism between closed intervals for all $j$. This type of process is called Denjoy blow-up (for instance, see the Construction 2.45 of \cite{Calebook}). We use $\tilde{R}$ to denote the new action obtained from $R$ as above. 

Now consider this circle as $\partial \HH^2$. For each $j$, connect the endpoints of $I_j$ by a geodesic of $\HH^2$. Then we get a lamination, and call it $\Lambda_R$, which is invariant under 
the cyclic group $G_R$ generated by $\tilde{R}$. Let $P_R$ be the unique complementary region of $\Lambda_R$ which does not contain any open arc of $S^1$. 
Then the following lemma holds. 

\begin{lem} 
\label{lem:irrationalrotation} 
 No $G_R$-invariant lamination meets the interior of $P_R$. 
\end{lem}
\begin{proof}
  Let $l$ be a leaf intersecting the interior of $P_R$. The $G_R$-action is semi-conjugate to $R$ via the monotone map $f : S^1 \to S^1$ which collapses each $I_j$, reverting the process of Denjoy blow-up.
  If the orbit closure of $l$ under the $G_R$-action gives a $G_R$-invariant lamination, then so does the orbit closure of $f(l)$ under $R$-action. Since $l$ intersects the interior of $P_R$, $f(l)$ is not degenerate. But $R$ cannot have any invariant lamination, since an irrational rotation maps any pair to a linked pair under some power of $R$, a contradiction. Hence the orbit closure of $l$ under the $G_R$-action cannot be a lamination. This implies that no invariant lamination of $G_R$ has a leaf intersecting the interior of $P_R$. 
\end{proof} 

$\Lambda_R$ is not a dense lamination. We can fix this by putting infinitely many copies of $\Lambda_R$ together in a nice way. 

Pick a leaf $l$ of $\Lambda_R$ and consider a larger group: the maximal orientation-preserving subgroup $G$ of the group $G' = <\tilde{R}, r(l)>$ generated by $\tilde{R}$ and the reflection $r(l)$ along the leaf $l$. Note that $G$ is simply $G' \cap \Homeop(S^1)$. 
 We claim that $G$ is strictly $\COL_1$. The orbit closure of $l$ under the $G'$-action is a dense lamination, call it $\Lambda(R)$. 
 The images of $P_R$ under the elements of $G'$ tesselate the open disk. 

Suppose there exists another $G$-invariant lamination $\hat{\Lambda}$ and let $L$ be a leaf of $\hat{\Lambda}$ which is not contained in $\Lambda(R)$. Then $L$ must intersect the interior of some gap $P$. But the action of $\Stab(P)$ is like the one of $G_R$ by construction where $\Stab(P) = \{\gamma \in G : \gamma(P) = P\}$. By Lemma \ref{lem:irrationalrotation}, the $\Stab(P)$-orbit of $L$ has linked elements so $\hat{\Lambda}$ cannot be a lamination. Hence $\Lambda(R)$ is the only invariant lamination of $G'$.

There are some questions we can ask. If we take a rotation $R'$ by a different irrational angle, are $\Lambda(R)$ topologically conjugate to $\Lambda(R')$? What can we say about the structure of the group $G$? 
It would be very interesting to know what makes the difference between strictly $\COL_1$ and $\COL_2$. 

\subsection{Strictly $\COL_2$ groups}
We shall now construct an example of a strictly $\COL_2$ group. 
Let $S$ be a closed orientable surface with genus $g \ge 2$. Thus it admits $\HH^2$ as its universal cover. Let $\phi: S \to S$ be a pseudo-Anosov homeomorphism from $S$ to itself. 
Since $\HH^2$ is simply connected, $\phi$ lifts to $\tilde{\phi} : \HH^2 \to \HH^2$. Since $\tilde{\phi}$ is a quasi-isometry, it extends continuously to $\partial \HH^2$. The restriction of this extension to the boundary circle gives a homeomorphism $\overline{\phi} : S^1 \to S^1$ where $S^1 = \partial \HH^2$. 
Let $G_\phi$ be the infinite cyclic subgroup of $\Homeop(S^1)$ generated by $\overline{\phi}$. 

It is well known that any pseudo-Anosov homeomorphism of a hyperbolic surface has a pair of transverse invariant laminations, the stable and unstable laminations. One can obtain them as limits of images of a simple closed curve under the forward and backward iterates of the pseudo-Anosov map. Let $\Lambda^{\pm}$ denote those two laminations on $S$ invariant  under $\phi$. Then these laminations lift to laminations $\tilde{\Lambda}^{\pm}$ in $\HH^2$ invariant under $\tilde{\phi}$. Then the endpoints of the leaves of $\tilde{\Lambda}^{\pm}$ form laminations $\overline{\Lambda}^{\pm}$ in $S^1$ invariant under $\overline{\phi}$. 

\begin{lem}
\label{lem:liftinglamdense} 
 $\overline{\Lambda}^{\pm}$ are dense in $S^1$. 
\end{lem} 
\begin{proof}
 It suffices to show that the endpoints of the lifts of any leaf of $\Lambda^{\pm}$ are dense in $S^1$. 
 This is obvious from the following easy observation. Let $\gamma$ be any leaf of $\overline{\Lambda}^{\pm}$. For arbitrary geodesic $l$ of $\HH^2$ 
 and for a half-space $H$ bounded by $l$, some fundamental domain of $S$ should intersect $H$. Hence one of the lifts of $\gamma$ intersects $H$ and then one end hits the arc of $\partial_\infty \HH^2$ bounded by the endpoints of $l$, on the same side as $H$. This shows that for arbitrary open interval of $S^1$, some leaf has one endpoint in there. 
\end{proof} 

\begin{prop}
  $G := \pi_1(S) \rtimes <\overline{\phi}>  $-action on $\partial_\infty \HH^2$ is strictly $\COL_2$.
\end{prop} 
\begin{proof}  
Let $\Lambda$ be an invariant lamination under $G$. Then, it projects down to a lamination on $S$ which is invariant under $\phi$. 
However, $\Lambda^+$ or $\Lambda^-$ are minimal and filling (meaning that every simple closed curve on $S$ intersects the lamination). Hence, the projected lamination on $S$ contains either $\Lambda^+$ or $\Lambda^-$ as a sub-lamination. In particular, $\Lambda$ cannot be transverse to both $\overline{\Lambda}^{+}$ and $\overline{\Lambda}^{-}$. 
\end{proof} 

We just saw that one can produce a large family of examples of strictly $\COL_2$ groups via pseudo-Anosov surface homeomorphisms. 


Note that we also saw that any group containing irrational rotations is an example of a strictly $\COL_0$ group. 
The results of this section prove the following proposition. 

\begin{prop}
\label{thm:strictlycols2n}
 $\CO \supsetneq \COL_1 \supsetneq \COL_2 \supsetneq \COL_3$.
\end{prop} 


\subsection{$\COL_\infty$ groups}

We have seen some examples of groups which have a very small number of invariant laminations. In the other extreme, there are groups which admit infinitely many invariant laminations. 

First we show 
\begin{prop}
\label{prop:colnnonempty}
 $\COL_\infty$ is nonempty.
\end{prop}
\begin{proof} 
Let $S$ be a closed hyperbolic surface. One can find infinitely many non-homotopic simple closed curves. In the homotopy class of each simple closed curve, there exists a unique simple closed geodesic. Identify the universal cover of $S$ with $\HH^2$. The lift of a simple closed geodesic becomes a geodesic lamination which is invariant under the action of $\pi_1(S)$. By the same argument as in the proof of Lemma \ref{lem:liftinglamdense}, its endpoints form a dense subset of $S^1$. Now we found an infinite family of dense invariant laminations of $S^1$ where the action of $\pi_1(S)$ is the restriction of the natural extension of the deck transformation action on $\partial_\infty \HH^2$. By  construction, they are obviously transverse to each other. 
\end{proof} 


As we saw in the proof of above proposition, a surface group admits an infinite collection of transverse dense invariant laminations. Hence they lie in $\COL_\infty $.
There are two natural questions to ask. 
\begin{ques} 
\label{ques:colinfty}
Is $\COL_\infty$ same as $\cap_n \COL_n$? 
\end{ques} 
\begin{ques} 
\label{ques:infinitelamsurfacegroup}
Are there examples of $\COL_\infty$ other than surface groups? 
\end{ques}

For  Question \ref{ques:colinfty}, if we add a condition to the definition of $\COL_n$ that the $n$ transverse dense invariant laminations are minimal, then the answer is YES. 

\begin{prop} 
Let $G$ be $\COL_n$ with minimal laminations for all $n$. Then $G$ is $\COL_\infty$. 
\end{prop} 
\begin{proof}
 Pick $n$ arbitrary. Then $G$ admits $n$ transverse dense minimal invariant laminations $\Lambda_1, \ldots, \Lambda_n$. We will show that there exists a minimal dense $G$-invariant lamination $\Lambda_{n+1}$ which is transverse to each $\Lambda_i$ for $i = 1, \ldots, n$. 
 
 Note that a minimal $G$-invariant lamination $\Lambda$ is simply the orbit closure of an arbitrary leaf of $\Lambda$ under the $G$-action. Hence, any two different minimal $G$-invariant laminations are transverse to each other. Since $G$ is $\COL_{n+1}$, there must a minimal dense invariant lamination which is different from $\Lambda_1, \ldots, \Lambda_n$. Thus, there exists a pair $l = (a,b)$ of points of $S^1$ which is not a leaf of any $\Lambda_i$ for $i = 1, 2, \ldots, n$, and the orbit closure of $l$ forms a minimal dense invariant lamination transverse to $\Lambda_1, \ldots, \Lambda_n$. This new lamination can be taken as $\Lambda_{n+1}$.  
 
 What we proved is that for any existing collection of pairwise minimal dense invariant laminations of $G$, we can add an extra minimal dense $G$-invariant lamination so that the new collection is still pairwise transverse. One obtains an infinite collection of pairwise transverse minimal dense $G$-invariant laminations by performing this process infinitely many times. 
 \end{proof} 
 
 In the proof, we need the minimality of the laminations in order to add a lamination to an existing collection. We suspect that Question \ref{ques:colinfty} has an affirmative answer in general, but could not prove it without the minimality assumption. 
 
 For Question \ref{ques:infinitelamsurfacegroup}, the answer is still YES. One can construct an example using Denjoy blow-up. In the subsequent sections, however, we will see that the situation is very different as long as one requires the invariant laminations to be very-full. 

\section{Laminations on the Hyperbolic Surfaces}
\label{sec:surfacelaminations}
In this section, we will study the laminations on hyperbolic surfaces. 


\begin{defn} A surface $S$ admitting a complete hyperbolic metric is called \textbf{pants-decomposable} if there exists a non-empty multi-curve $X$ on $S$ consisting of simple closed geodesics so that the closure of each connected component of the complement of $X$ is a pair of pants. The fundamental group of some pants-decomposable surface is called a pants-decomposable surface group. The multi-curve $X$ used in the pants-decomposition of $S$ will be called a \textbf{pants-curve}. 
\end{defn}

Note that all hyperbolic surfaces of finite area except the thrice-puncture sphere are pants-decomposable. The thrice-puncture sphere is excluded by the definition, since we required the existence of a `non-empty' pants-curve. For a hyperbolic surface with infinite area, we still have a similar decomposition but some component of complement of the closure of a multi-curve could be a half-annulus or a half-plane.  For the precise statement, we refer to Theorem 3.6.2. of \cite{Hubb}. 

\begin{lem}
\label{lem:pseudoanosov}
Let $S$ be a hyperbolic surface of finite area which is not the thrice punctured sphere. For any pseudo-Anosov homeomorphism $f$ of $S$ and two arbitrary finite sets of simple closed curves $F_1, F_2$, there exists a large enough $n$ such that no curve in $F_1$ is homotopic to a curve in $f^n(F_2)$, where $f^n$ is the $n$th iterate of $f$. 
\end{lem}
\begin{proof} This is an immediate consequence of the fact that a pseudo-Anosov map has no reducible power. 
\end{proof} 

\begin{prop} 
\label{prop:infinitepantscurves}
Let $S$ be a pants-decomposable surface. Then there exists pants-curves $X_0, X_1, X_2$ so that no curve in $X_i$ is homotopic to a curve in $X_j$ for all $i \neq j$. 
\end{prop}
\begin{proof} 
Note that this claim is clear if $S$ is of finite area. We take an arbitrary pants-curve $X_0$ and a pseudo-Anosov map $f: S \to S$. Then by Lemma \ref{lem:pseudoanosov}, there exist large enough positive integers $n_1, n_2$ such that $X_0, X_1 = f^{n_1}(X_0), X_2 = f^{n_2}(X_0)$ are such pants-curves.  But we have no well-undestood notion of pseudo-Anosov map for an arbitrary surface of infinite area. 

Let $S$ be a pants-decomposable surface of infinite area. 

First we take an arbitrary pants-curve $X_0$. Seeing $X_0$ as some set of simple closed geodesics, choose a subset $B$ of $X_0$ such that no two curves in $B$ are boundary components of a single pair of pants, and each connected component of $S \setminus B$ is a finite union of pairs of pants, ie., of finite area. Let $(S_i)_{i \in \NN}$ be the enumeration of the connected components of $S \setminus B$. For each $i$, choose a pseudo-Anosov map $f_i$ on $S_i$ rel $\partial S_i$. By Lemma  \ref{lem:pseudoanosov}, there exists $n_i, m_i \in \NN$ so that $X_0 \cap S_i, f^{n_i}(X_0 \cap S_i), f^{m_i}(X_0 \cap S_i)$ are desired pants-curves on $S_i$.  
 
 Let $X_1 := B \cup (\cup_{i \in \NN} f_i^{n_i}(X_0 \cap S_i)), X_2 := B \cup (\cup_{i \in \NN} f_i^{m_i}(X_0 \cap S_i))$. We are not quite done yet, since all $X_0, X_1, X_2$ contain $B$. For each curve $\gamma$ in $B$, we choose a simple closed curve $\delta(\gamma)$ as in Figure \ref{fig:coloringonpantscurve}. They show three different possibilities for $\gamma$ as red, blue, and green curves, and in each case, $\delta(\gamma)$ is drawn as the curve colored in magenta. By definition of $B$, $\delta(\gamma)$ is disjoint from $\delta(\gamma')$ for $\gamma \neq \gamma' \in B$. Let $D$ be the positive multi-twist along the multi-curve $Y = \cup_{\gamma \in B} \delta(\gamma)$. Let $X_1' = D(X_1)$. A curve in $X_1$ which had zero geometric intersection number with $Y$ remains unchanged, and clearly it has no homotopic curves in $X_0$. A curve in $X_1$ which had non-zero intersection number with $Y$ now has positive geometric intersection number with $B$. Since no curve in $X_0$ has positive geometric intersection number with $B$, we are done for this case too. 
 
 Changing $X_2$ is a bit trickier. Let $(b_i)_{i \in \NN}$ be an enumeration of the curves in $B$. Let $D_i$ be the positive Dehn twist along $\delta(b_i)$. One can take $k_i$ for each $i$ so that each curve in $D_i^{k_i}(X_2)$ either is the image of a curve in $X_2$ which has zero geometric intersection number with $\delta(b_i)$ (so remains unchanged) or has positive geometric number with $b_i$ which is strictly larger than $2$. Since $\delta(b_i)$ are disjoint, the infinite product $D' := \prod_{i \in \NN} D_i^{k_i}$ is well-defined. Define $X_2'$ as $X_2' = D'(X_2)$. Note that the geometric intersection number between a curve $\gamma$ in $X_1'$ and $b_i$ for some $i$ is at most $2$.  Now it is clear that $X_0, X_1', X_2'$ are desired pants-curves. 
\end{proof} 

\begin{figure}	
	\centering
	\begin{subfigure}[t]{1.7in}
		\centering
		\includegraphics[scale=0.4]{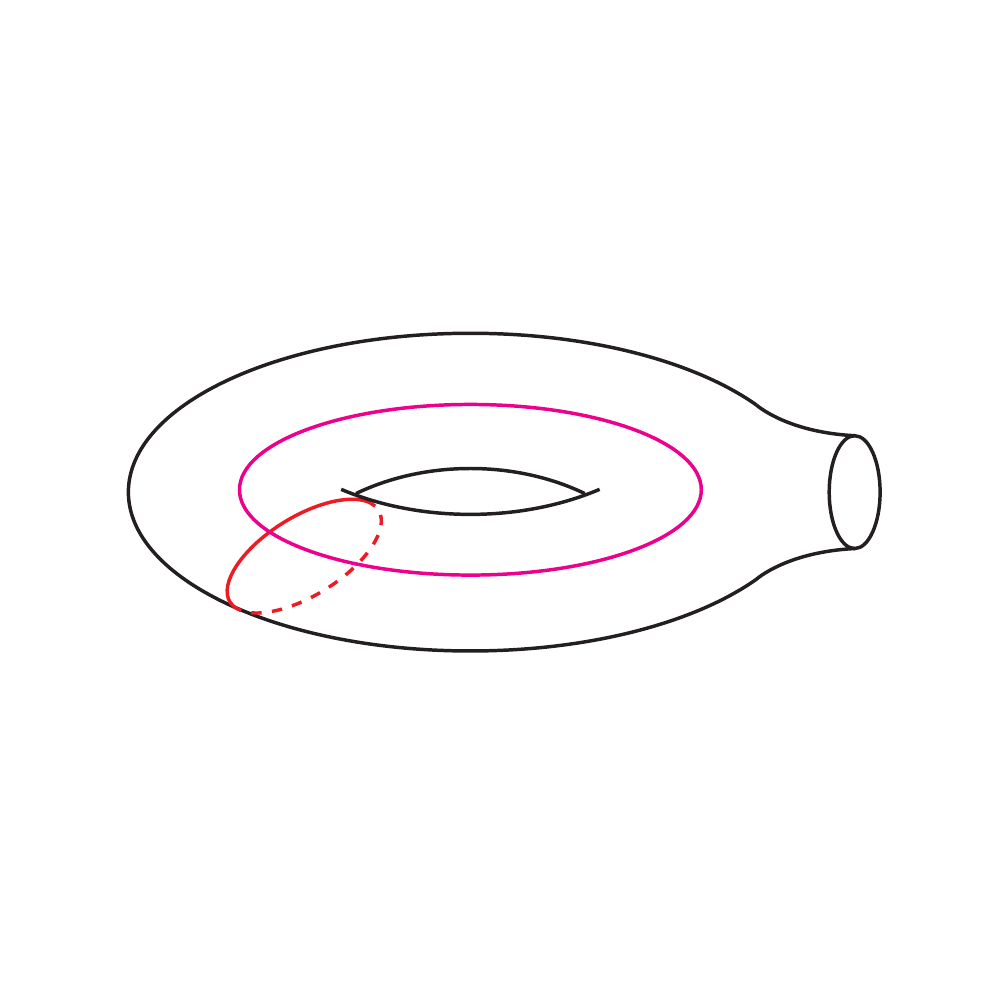}
		\caption{Boundary component of a single pair of pants}\label{fig:pants1}		
	\end{subfigure}
	\quad
	\begin{subfigure}[t]{1.7in}
		\centering
		\includegraphics[scale=0.4]{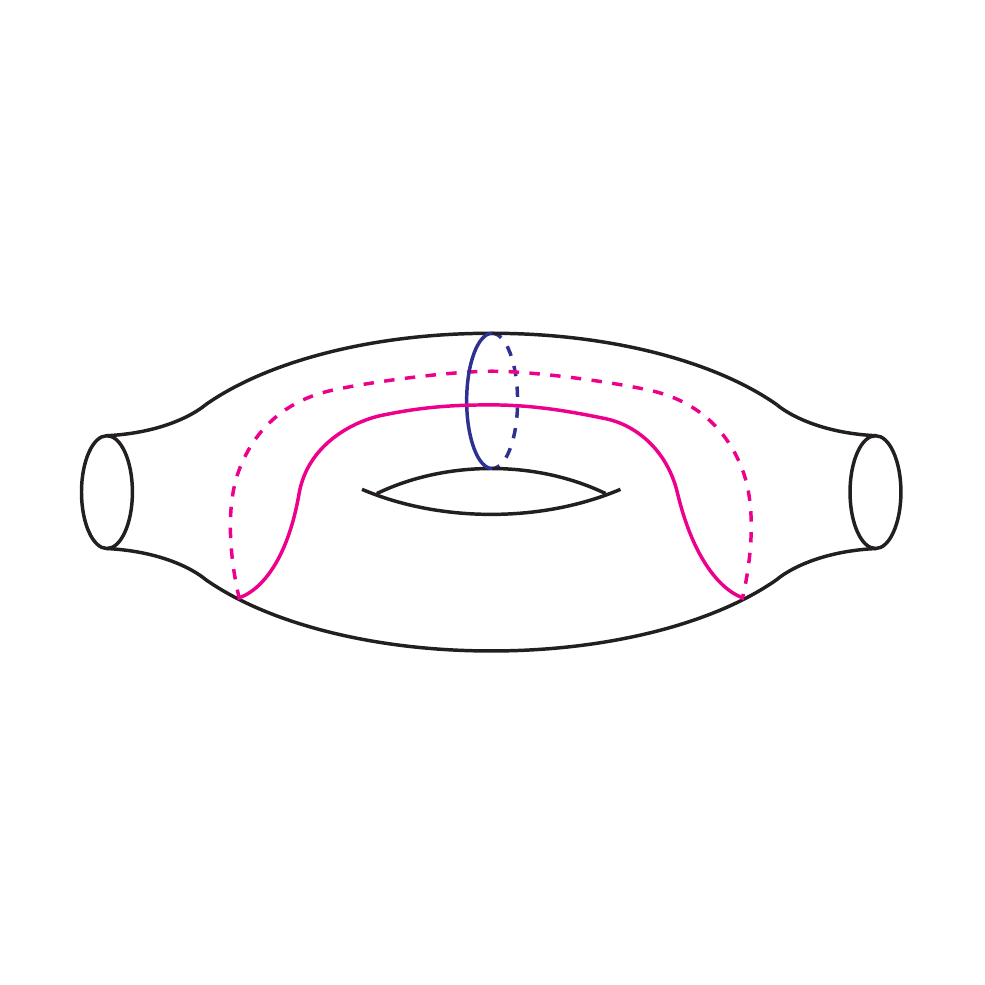}
		\caption{One of two shared boundary components of two pairs of pants}\label{fig:pants2}
	\end{subfigure}
	\quad
	\begin{subfigure}[t]{1.7in}
		\centering
		\includegraphics[scale=0.4]{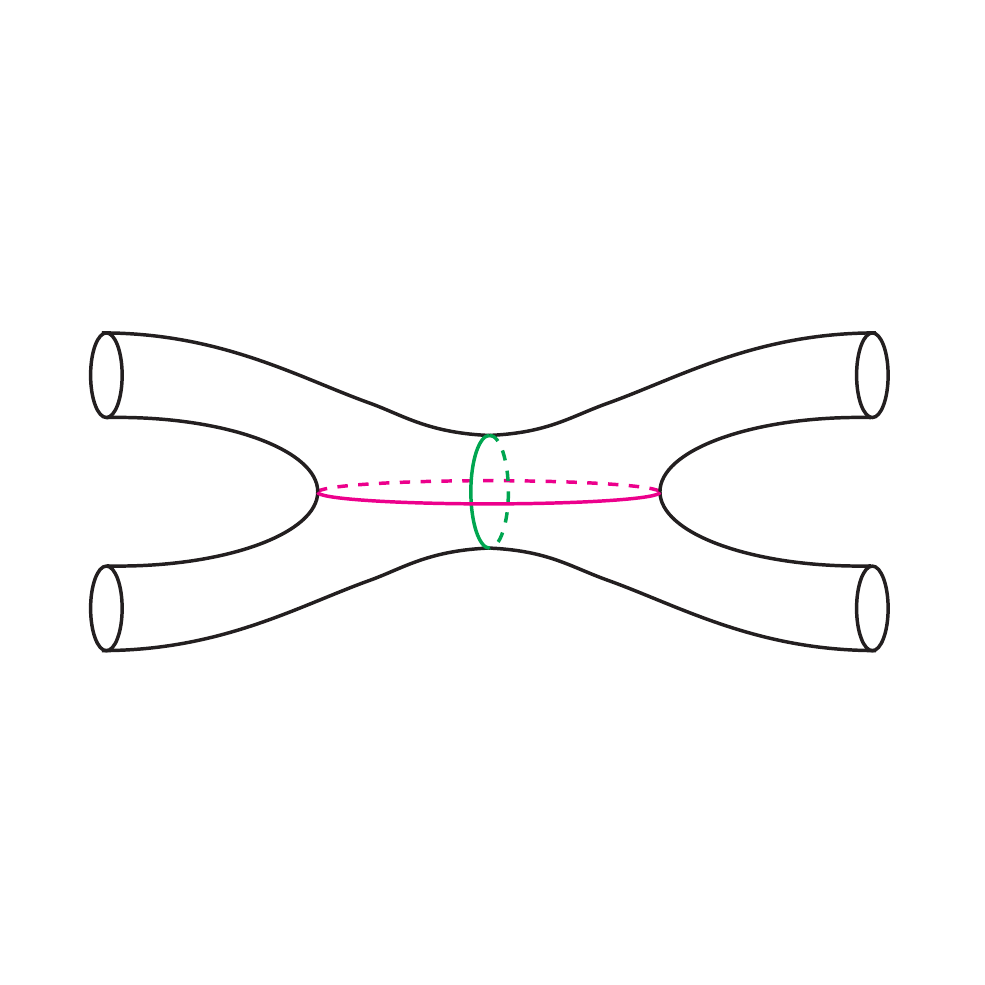}
		\caption{Unique shared boundary component of two pairs of pants}\label{fig:pants3}
	\end{subfigure}
	\caption{This shows how to choose the multi-curve along which we will perform the postive multi-tiwst to produce a new pants-curve.}\label{fig:coloringonpantscurve}
\end{figure}

Next two lemmas are preparation to produce a pants-like collection of laminations out of the pants-curve we produced above. 

\begin{lem} 
\label{lem:hyperbolicend}
Let $G$ be a $\COL$ group with an invariant lamination $\Lambda$ and $g \in G$ be a hyperbolic element. If $\Lambda$ has leaf $l$ one end of which is fixed by $g$, then $\Lambda$ has a leaf joining two fixed points of $g$. 
\end{lem}
\begin{proof}
Either $g^n(l)$ or $g^{-n}(l)$ converges to the axis of $g$ as $n$ goes to $\infty$. 
\end{proof} 

\begin{lem} 
\label{lem:fixedpointisnotendpoint}
Let $G$ be a $\COL_n$ group for some $n \ge 1$ and let $\{\Lambda_{\alpha}\}$ be a collection of $n$ pairwise transverse dense invariant laminations of $G$. 
If $x \in S^1$ is a fixed point of a hyperbolic element $g$ of $G$, then there exists at most one lamination $\Lambda_\alpha$ which has a leaf with $x$ as an endpoint. 
\end{lem} 
\begin{proof} This is a consequence of Lemma \ref{lem:hyperbolicend} and the transversality of the laminations. 
\end{proof}

\begin{figure}	
	\centering
	\begin{subfigure}[t]{1.7in}
		\centering
		\includegraphics[scale=0.4]{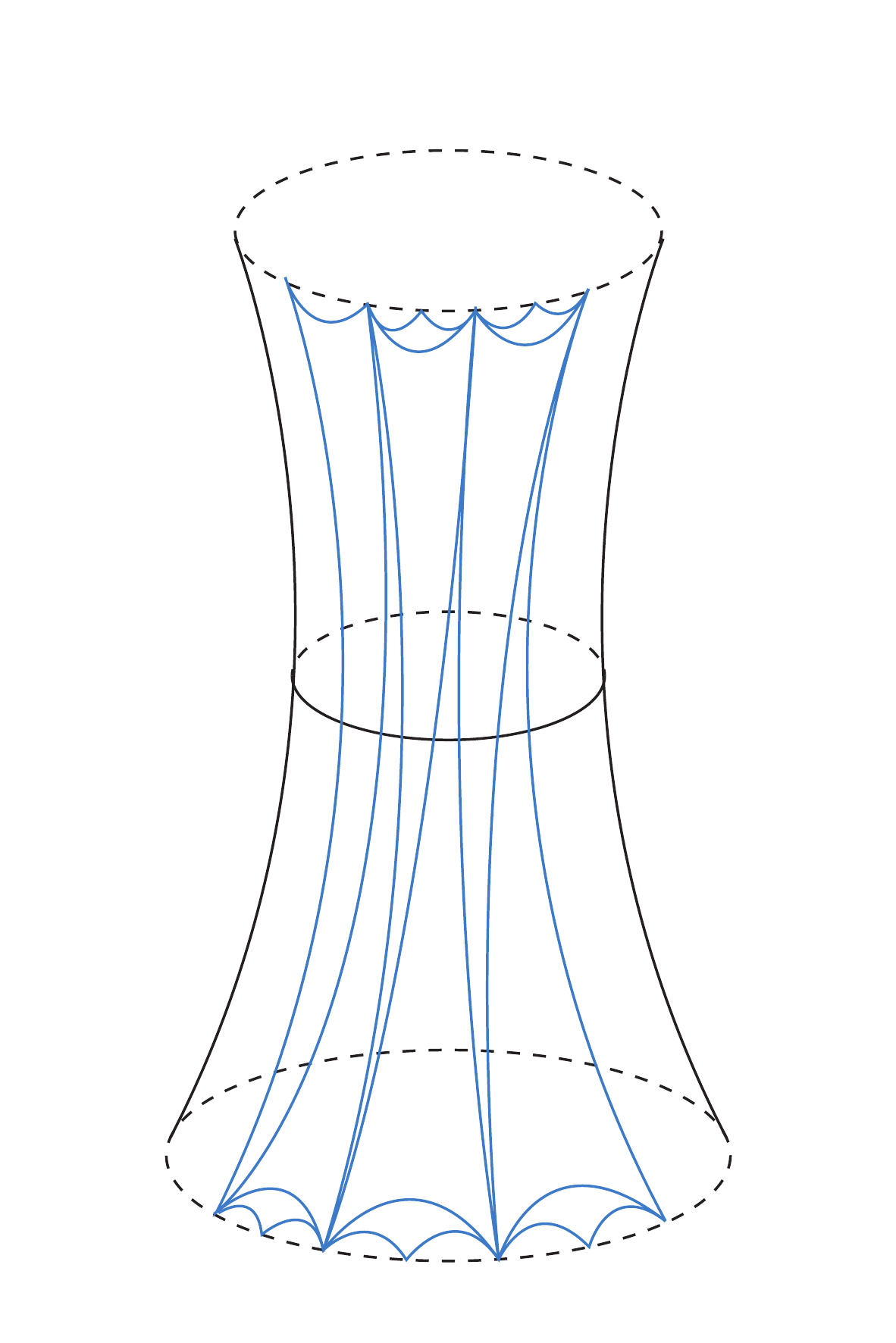}
		\caption{This is the case when the quotient surface is an infinite annulus. }\label{fig:annulus1}		
	\end{subfigure}
	\quad
	\begin{subfigure}[t]{1.7in}
		\centering
		\includegraphics[scale=0.4]{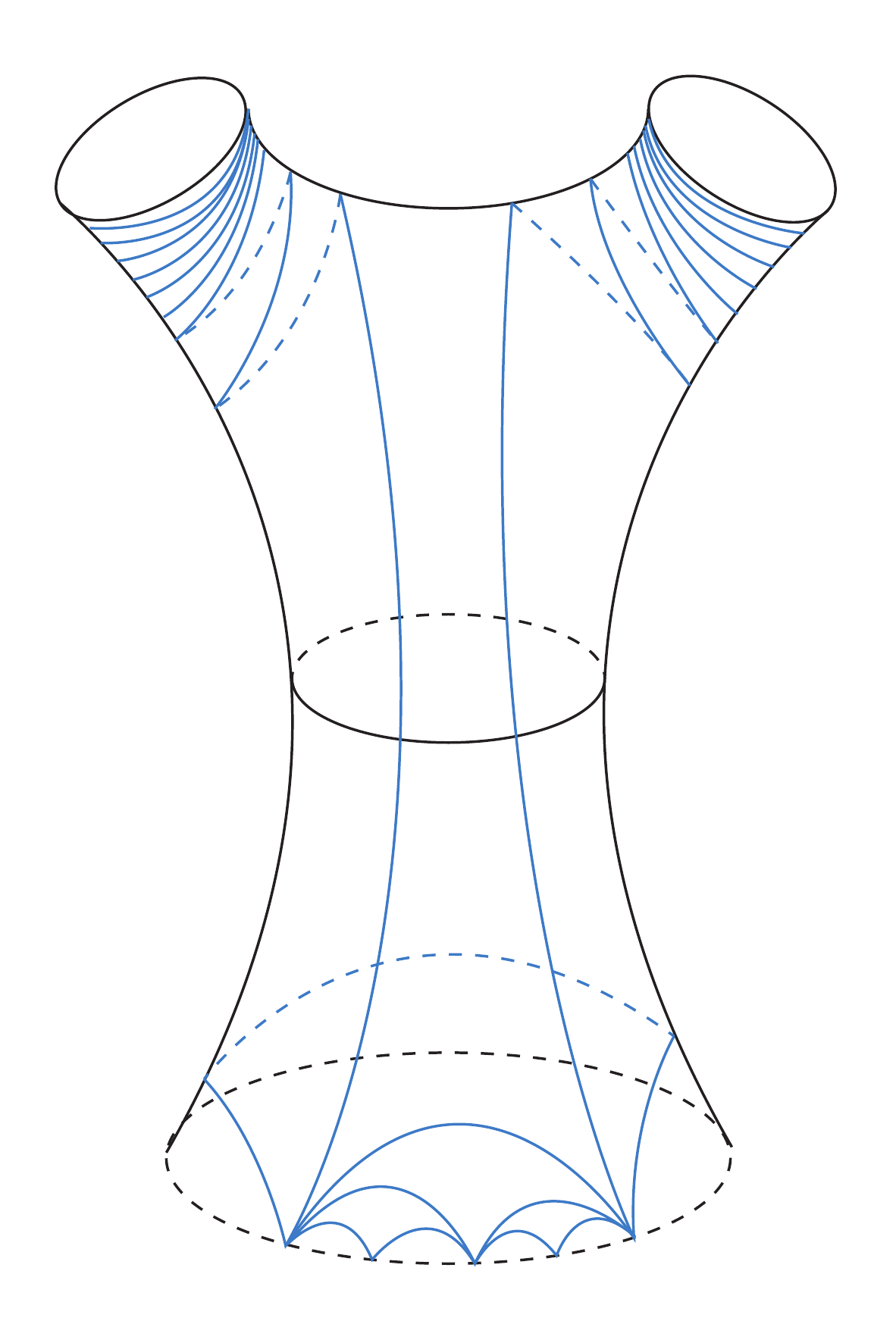}
		\caption{This is the case when there exists an annulus component glued to a pair of pants. }\label{fig:annulus2}
	\end{subfigure}
	\caption{ Indeed one can put arbitrarily many pairwise transverse very full laminations on the annulus components in this way. }
	\label{fig:laminatedannuli}
\end{figure}

\begin{figure}[ht]
\begin{center}
\includegraphics[scale=0.6]{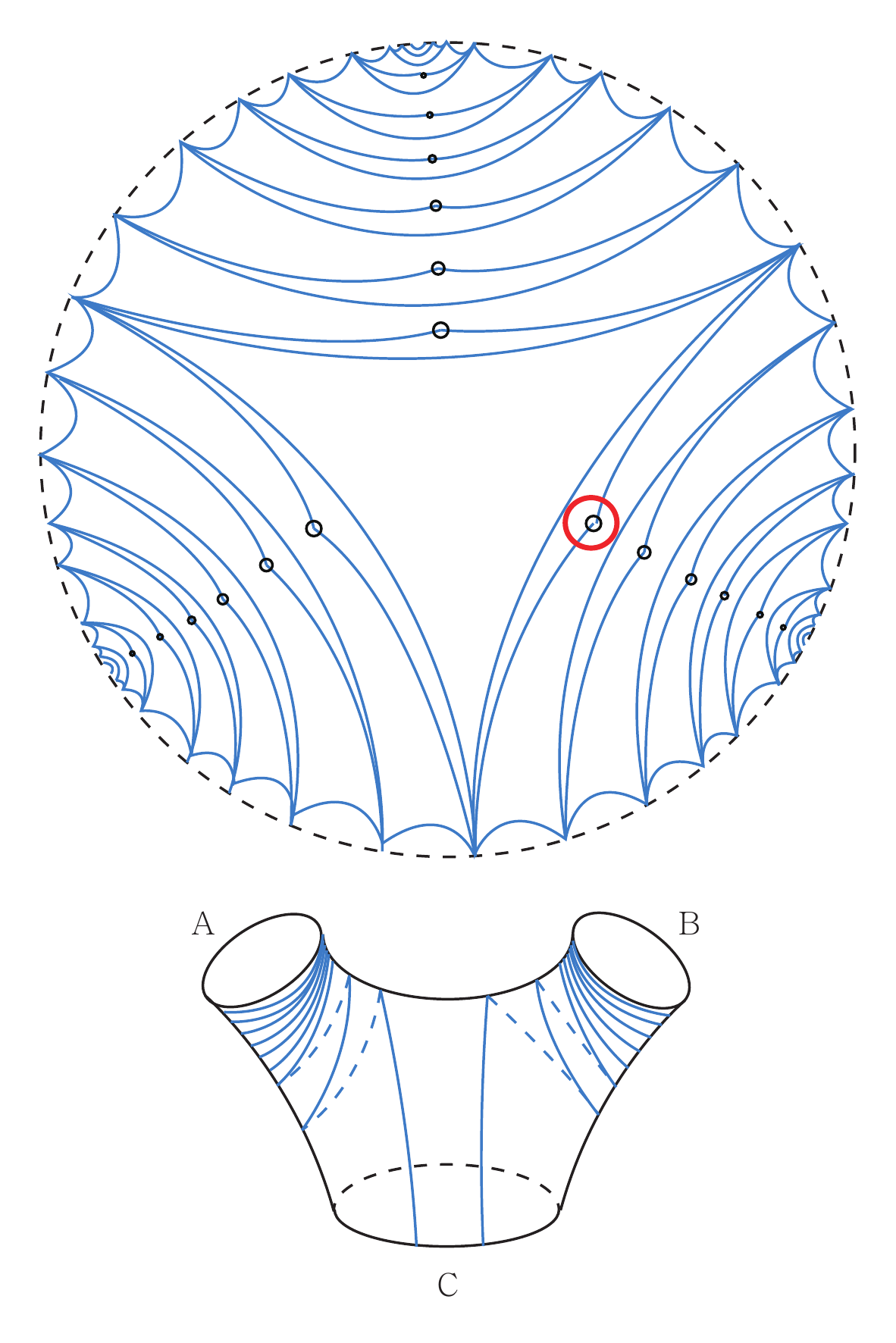}
\end{center}
\caption{ Above figure shows an example of the subsurface which we consider in the proof of Theorem \ref{thm:fuchsianispcol3}. One can choose the endpoints of the leaves on the ideal boundary arbitrarily so that we can put as many pairwise transverse very-full laminations on such a subsurface as we want. Look at the part where the red circle is. Here a pair of pants is attached as in the figure below. The boundary component labeled by `C' is not included in $X_0$ but those labeled by `A' and `B' are. The boundary curves A and B could be cusps or glued along each other. 
}
\label{fig:laminatedhalfplanewithpants}
\end{figure}

\begin{thm}
\label{thm:fuchsianispcol3}
 Any Fuchsian group $G$ such that $\HH^2/G$ is not the thrice-punctured sphere is a pants-like $\COL_3$ group. 
\end{thm}
\begin{proof} 
 We start with the case when $G$ is the fundamental group of a pants-decomposable surface $S$. 
 Let $(X_i)_{i = 0,1,2}$ be the pants-curves as in Proposition \ref{prop:infinitepantscurves}. For each $i$, let $L_i$ be the lamination on $S$ obtained from $X_i$ by decomposing the interior of each pair of pants into two ideal triangles. It is possible to put a hyperbolic metric on $S$ so that $L_i$ is a geodesic lamination. Identify the universal cover of $S$ with $\HH^2$. $G$ acts on the circle at infinity.
 Let $\Lambda_i$ be the lamination of the circle at infinity obtained by lifting $L_i$ to $\HH^2$ and taking the end-points data. Since all the complementary regions are ideal triangles, it is very full.
 Also any leaf of $L_i$ is either a simple closed geodesic, or it is an infinite geodesic each of whose end either accumulates to a simple closed geodesic or escape to a cusp. Hence each end is a fixed point of some parabolic or hyperbolic element of $G$. Now the pants-like property follows from Lemma \ref{lem:hyperbolicend} and the transversality of the laminations. 
 
 We would like to get the same conclusion as before in the general case if $G$ is a Fuchsian group but its quotient surface $\HH^2/G$ is neither a thrice-punctured sphere nor pants-decomposable. 
 
 Let's first deal with the half-annulus components. Suppose that $X$ is a multi-curve on the quotient surface $S$ such that $S \setminus X$ consists of pairs of pants and half-annuli. 
 If two half-annuli are glued along a simple closed geodesic, our surface is actually an annulus and the lamination could be taken as in Figure \ref{fig:annulus1}. Since we can take the ends of such a lamination arbitrarily, it is obvious that there are arbitrarily many such invariant laminations which are pairwise transverse. If the surface is not an annulus, a half-annulus component needs to be attached to a pair of pants. Let $X_0$ be the collection of simple closed geodesics obtained from the $X$ by removing those boundaries of half-annulus components. Let $S'$ be the complement of the half-annuli. As in the proof of Proposition \ref{prop:infinitepantscurves}, we can find other pants-curves $X_1, X_2$ on $S'$ so that $X_0, X_1, X_2$ are disjoint in the curve complex of $S'$. Now we decompose the interior of each pair of pants into two ideal triangles as before. 
 
 We need to put more leaves on each component of $S \setminus X_i$ for any $i$ which is the union of one half-annulus and one pair of pants glued along a cuff.  We construct a lamination inside such a component as in Figure \ref{fig:annulus2}. Again, we can put an arbitrary lamination on the ideal boundary part of the half-annulus. Note that we construct each lamination so that all gaps are finite-sided, thus we are done. 
 
 Now we consider the case where $S$ has even half-plane components. In the Theorem 3.6.2. of \cite{Hubb}, it is also shown that if $Z$ denotes the set of points of a pants-curve $X$, then components of $\overline{Z} \setminus Z$ are simple infinite geodesics bounding half-planes, i.e., we know exactly how the half-plane components arise in the decomposition of a complete hyperbolic surface. Let $X$ be a multi-curve and let $Z$ be the set of points of simple closed curves in $X$ such that $S \setminus \overline{Z}$ consists of pairs of pants, half-annuli, and half-planes, and the boundaries of half-plane components form the set $\overline{Z} \setminus Z$. We will define $X_0$ by removing some geodesics from $X$. As before,we remove all the boundary curves of half-annulus components. Observe that there is a part of a surface which is homeomorphic to a half-plane with families of cusps and geodesic boundaries which converge to the ideal boundary (see Figure 3.6.3 on the page 86 of \cite{Hubb} for example). On this subsurface, there are infinitely many components of $\overline{Z}$ so that this subsurface is decomposed into pairs of pants and some half-planes. We remove all the components of $X$ appear on this type of subsurface. Again, $X_0$ is a pants-curve of a pants-decomposable subsurface $S'$ of $S$ with geodesic boundaries. On $S'$, we construct $X_1, X_2$ as before. 
 Among the connected components of $S \setminus S'$, the one containing a half-annulus can be laminated as we explained in the previous paragraph. In the connected component which is homeomorphic to an open disk with punctures, we can do this as in Figure \ref{fig:laminatedhalfplanewithpants}. Once again, since the ideal boundary part is invariant, we can put an arbitrary lamination there. It is also obvious that the way we construct a lamination gives a very full lamination. 
 
 We have shown the theorem. 
\end{proof} 

\begin{rmk} We constructed a pants-like collection of laminations for Fuchsian groups using pants-decompositions in the proof of Theorem \ref{thm:fuchsianispcol3}. This is where the name `pants-like' comes from. 
\end{rmk}

We would like to see if the converse of Theorem \ref{thm:fuchsianispcol3} is also true. In order to answer that question, one needs to analyze the properties of pants-like $\COL_3$ groups. 


\section{Rainbows in Very-full Laminations}
\label{sec:rainbow}
Before we move on, we would like to understand better the structure of very-full laminations. Recall that $\mathcal{M}$ is the set of all pairs of two distinct points of $S^1$, which is homeomorphic to an open M\"{o}bius band.

Let $p \in S^1$ and $\Lambda$ be a dense lamination on $S^1$. Suppose that there is a sequence of leaves of $\Lambda$ both of whose ends converge to $p$ but from opposite sides. We call such a sequence a \emph{rainbow} at $p$. Imagine the upper half-plane model of $\HH^2$ and that we stand at $x$ in the real line which is not an endpoint of a lamination. The name `rainbow' would make sense in this picture. See Figure \ref{fig:rainbow}.

\begin{figure}[ht]
\begin{center}
\includegraphics[scale=0.5]{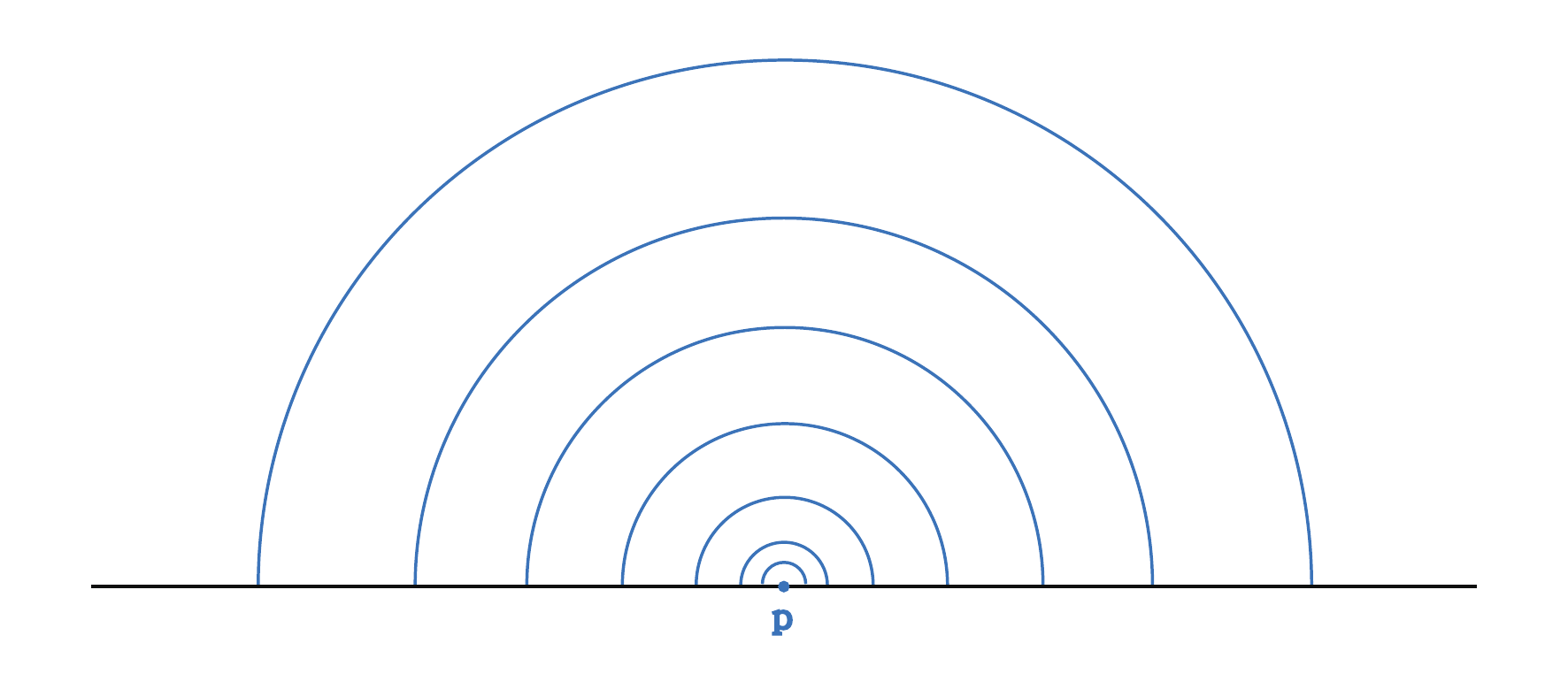}
\end{center}
\caption{ This is a schematic picture of a rainbow at $p$. }
\label{fig:rainbow}
\end{figure}  

The following lemma is more or less an observation. 
\begin{lem} 
\label{lem:veryfulllam}
Let $\Lambda$ be a very full lamination of $S^1$. Then $\Lambda$ is dense. Further, for any gap $P$ of $\Lambda$, if $x \in S^1 \cap P$, then $x$ is an endpoint of some leaf of $\Lambda$. 
\end{lem}
\begin{proof} Suppose $\Lambda$ is not dense. Then we can take an open connected arc $I$ of $S^1$ where the leaves of $\Lambda$ have no endpoints. Let $l$ be a geodesic connecting the endpoints of $I$. $\Lambda$ has no leaf intersecting $l$. Take a point $p$ on $l$. Clearly the gap containing $p$ cannot be a finite-sided ideal polygon. 
 Suppose $P$ is a gap of $\Lambda$ and $x \in S^1 \cap P$. Since $P$ is a finite-sided ideal polygon, it intersects $S^1$ only at points to which two sides of $P$ converges. Hence $x$ is an endpoint of some leaf. 
\end{proof} 

The proof of the following lemma is easily provided from basic facts of hyperbolic geometry. 
\begin{lem} 
\label{lem:endpointcriterion}
Consider a very full lamination $\Lambda$ of $S^1$. Let $x \in \DD$. For $p \in S^1$, a gap of $\Lambda$ containing $x$ contains $p$ if and only if there is no leaf of $\Lambda$ crossing the geodesic ray from $x$ to $p$. 
\end{lem}

Recall that for a lamination $\Lambda$ on $S^1$, $E_\Lambda$ denotes the set of endpoints of the leaves of $\Lambda$. There is a nice dichotomy. 
\begin{thm} [There are enough rainbows]
\label{thm:rainbow}
 Let $\Lambda$ be a very full lamination of $S^1$. For $p \in S^1$, either $p$ is in $E_\Lambda$ or $p$ has a rainbow. These two possibilities are mutually exclusive. 
\end{thm}
\begin{proof}
 It is clear that if $p \in E_\Lambda$, there is no rainbow. Suppose there is no rainbow for $p$. Then $p$ has a neighborhood $U$ so that if a leaf of $\Lambda$ has both endpoints in $U$, then both endpoints are contained in the same connected component of $U \setminus \{p\}$. Replacing $U$ by a smaller neighborhood, we may assume that no leaf connects the endpoints of $U$. 
 
 Identify $S^1$ with the boundary of the hyperbolic plane $\DD$ and realize $\Lambda$ as a geodesic lamination on $\DD$.  Let $q_1$ be a point on the geodesic connecting the endpoints of $U$. We may assume that there is no leaf passing through $q_1$.  The only way not having such a point is that the entire $\DD$ is foliated and $p$ is an endpoint of one of the leaves. Let $L$ be the geodesic passing through $q_1$ and ending at $p$ (See Figure \ref{fig:norainbow}). We denote the part of $L$ between a point $x$ on L and $p$ by $L_x$. Note that any leaf of $\Lambda$ crossing $L_{q_1}$ has one end in $U$ so that the other end must be outside $U$ by the assumption on $U$. 
 
 If there is no leaf of $\Lambda$ intersecting $L_{q_1}$, then the gap containing $q_1$ contains $p$ by Lemma \ref{lem:endpointcriterion}. Hence  $p$ must be an endpoint of a leaf by Lemma \ref{lem:veryfulllam} and we are done. 
 Suppose there is a leaf $l_1$ which crosses $L_{q_1}$ at $x$. Then let $q_2$ be a point in $L_x$ so that there is no leaf of $\Lambda$ passing through it. If there was no such a $q_2$, it means there is a leaf passing through each point of $L_x$ so there must be a leaf ending at $p$ whose other end is necessarily outside $U$. So, we may assume that such $q_2$ exists. 
 
 If there is no leaf of $\Lambda$ crossing $L_{q_2}$, then the gap containing $q_2$ contains $p$, we are done. Otherwise, a leaf, say $l_2$, crosses $L_{q_2}$ at $x_2$. Repeat the process until we obtain an infinite sequence $(q_i)$ on $L$ which converges to $p$. This is possible, since otherwise we must have some $x$ on $L$ such that no leaf of $\Lambda$ crosses $L_x$.  
  Since $(q_i)$ converges to $p$, the endpoints of the sequence $(l_i)$ of leaves in $U$ form a sequence converging to $p$, and the other endpoints are all outside $U$. By the compactness of $S^1$, we can take a convergent subsequence so that $p$ is an endpoint of the limiting leaf. In any case, $p$ must be an endpoint of a leaf. 
 
 Therefore, $p \in E_\Lambda$ if and only if $p$ has no rainbow. 
 \end{proof} 
 
 \begin{figure}[ht]
\begin{center}
\includegraphics[scale=0.5]{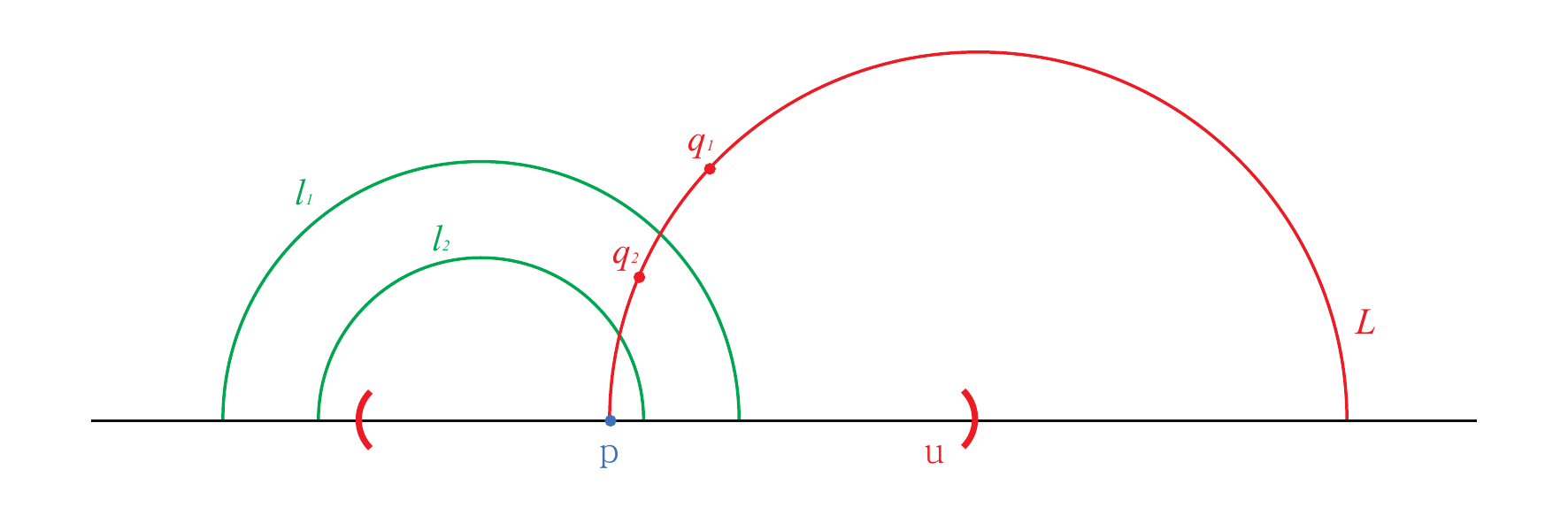}
\end{center}
\caption{ This shows a situation when we have no rainbow for $p$. }
\label{fig:norainbow}
\end{figure}  

\begin{cor}
\label{cor:parabolicfixedpoints}
Let $G$ be a group acting on $S^1$ and $\Lambda$ be a $G$-invariant very-full lamination. 
For $x \in S^1$ which is the fixed point of a parabolic element $g$ of $G$, there exist infinitely many leaves which have $x$ as an endpoint. 
\end{cor}
\begin{proof} 
Let $x \in S^1$ be the fixed point of a parabolic element $g$ of $G$ and pick $\Lambda_\alpha$. By Theorem \ref{thm:rainbow}, if $x$ is not an endpoint of a leaf of $\Lambda_\alpha$, then $x$ has a rainbow. But any leaf none of whose ends is $x$ must be contained in a single fundamental domain of $g$ to stay unlinked under the iterates of $g$ (here, a fundamental domain is the arc connecting $y$ and $g(y)$ in $S^1$ for some $y$ different from $x$). Then the existence of a rainbow would imply that $g$ is a constant map whose image is $x$ but  it is impossible since $g$ is a homeomorphism. Hence $x$ is an endpoint of some leaf $l$ of $\Lambda_\alpha$. 
Then $(g^n(l))_{n \in \ZZ}$ are infinitely many distinct leaves of $\Lambda_\alpha$ all of which have $x$ as an endpoint. 
\end{proof} 

\begin{cor}[Boundary-full laminations] 
\label{cor:boundaryfull}
 Suppose a group $G$ acts on $S^1$ faithfully and minimally. Let $\Lambda$ be a lamination of $S^1$ invariant under the $G$-action. If $\Lambda$ is very full and totally disconnected, then $\Lambda$ is boundary-full. 
\end{cor}
\begin{proof} 
 The minimality of the action implies that once the closure of the lamination in $\overline{\mathcal{M}}$ contains at least one point in $\partial \mathcal{M}$, then it contains $\partial \mathcal{M}$ and thus the lamination is very-full (this is a simple diagonalization argument). 
 
 Let $l_1$ be any leaf of $\Lambda$. Due to the minimality, some element of $G$ maps one of the ends of $l_1$ somewhere in the middle of the shortest two arcs joining the endpoints of $l$. Let $l_2$ be the image of $l_1$ under the action of this element. Again due to the minimality, one can find an element of $G$ which maps one of the ends of $l_2$ somewhere in the middle of the shortest arcs in the complement of the endpoints of $l_1$ and $l_2$ in $S^1$. Let $l_3$ be the image of $l_2$ under that element. Repeating this procedure, one gets a sequence $(l_n)$ of leaves for which the distance between their endpoints tends to zero, hence giving a desired point in $\partial M$. 
\end{proof} 

In fact, the laminations we constructed for pants-decomposable surface groups satisfy the hypotheses of Corollary \ref{cor:boundaryfull}. Hence all of them are boundary-full laminations. 

\section{Classification of Elements of Pants-like $\COL_3$ Groups}
\label{sec:pcol3}

Any element of a Fuchsian group has at most two fixed points on $\partial_\infty \HH^2$. Hence, it might be useful to check how many fixed points an element of a pants-like $\COL_3$ group can have. 

\begin{lem}
\label{lem:leafconnectingfixedpoints}
Let $f$ be a non-identity orientation-preserving homeomorphism of $S^1$ with $3 \le |\Fix_f|$. Then any very full lamination $\Lambda$ invariant under $f$ has a leaf connecting two fixed point of $f$. Moreover, for any connected component $I$ of $S^1 \setminus \Fix_f$ with endpoints $a$ and $b$, at least one of $a$ and $b$ is an endpoint of a leaf of $\Lambda$. 
\end{lem}
\begin{proof} 
 Let $I$ be a connected component of $S^1 \setminus \Fix_f$ with endpoints $a$ and $b$. Since $\Fix_f$ has at least three points, one can take $c \in \Fix_f \setminus \{a, b\}$. Relabeling $a$ and $b$ if necessary, we may assume that the triple $a, b, c$ are counterclockwise oriented. 
 
 Suppose $a$ is not an endpoint of a leaf of $\Lambda$. Then there exists a rainbow in $\Lambda$ at $a$ by Theorem \ref{thm:rainbow}. In particular, there exists a leaf $l$ such that one end of $l$ lies in $I$ and the other end lies outside $I$, call the second one $d$.  If $d$ is a fixed point of $f$, then replace $c$ by $d$. Otherwise, we may assume that $a, c, d$ are counterclockwise oriented and there is no fixed point of $f$ between $c$ and $d$ after replacing $c$ by another fixed point if necessary. Clearly, either $f^n(l)$ or $f^{-n}(l)$ converges to the leaf connecting $b$ and $c$ (may not be the same $c$ as the $c$ at the beginning). 
This proves the lemma. 
\end{proof} 

\begin{cor}
\label{cor:atmost2fixedpoints}
Let $G$ be a pants-like $\COL_3$ group. Then for any $g \in G$, one must have $|\Fix_g| \le 2$. 
\end{cor}
\begin{proof} Let $\{ \Lambda_1, \Lambda_2, \Lambda_3 \}$ be a pants-like collection of $G$-invariant laminations. Suppose that there exists an element $g$ of $G$ which has at least three fixed points on $S^1$. Let $I$ be a connected component of $S^1 \setminus \Fix_g$ with endpoints $a$ and $b$. Then by Lemma \ref{lem:leafconnectingfixedpoints}, each of $a$ and $b$ is an endpoint of a leaf of some $\Lambda_i$. Hence, if none of $a, b$ is the fixed point of a parabolic element of $G$, we get a contradiction to the pants-like property. 

 Suppose $a$ is the fixed point of a parabolic element $h \in G$. By Corollary \ref{cor:parabolicfixedpoints} (or rather the proof of it), there must be a leaf $l$ of $\Lambda_i$ for any choice of $i$ such that one end of $l$ is $a$ and the other end lies in $I$. Then either $g^n(l)$ or $g^{-n}(l)$ converges to the leaf connecting $a$ and $b$ as $n$ increases. Hence each $\Lambda_i$ must have the leaf connecting $a$ and $b$, contradicting to the transversality. Similarly $b$ cannot be a cusp point either. This completes the proof. 
\end{proof}

\begin{lem} 
\label{lem:separatinghyperbolicfixedpoints}
Let $G$ be a group acting on $S^1$ and $\Lambda$ be a very-full $G$-invariant lamination. 
For each hyperbolic element $g \in G$ with fixed points $a$ and $b$, if $\Lambda$ does not have $(a,b)$ as a leaf, there must be leaf $l$ of $\Lambda_{\alpha}$ which separate $a, b$ (ie., not both endpoints of $l$ lies in the same connected component of $S^1 \setminus \{a, b\}$). 
\end{lem} 
\begin{proof}
This is just an observation using the existences of a rainbow. 
\end{proof} 

\begin{lem} 
\label{lem:parabolichyperbolic} 
 Let $G$ be a pants-like $\COL_3$ group. If $f \in G$ is parabolic, then its fixed point is a parabolic fixed point, ie., the fixed points behaves as a sink on the one side and as a source on the other side. If $f$ is hyperbolic, it has one attracting and one repelling fixed points, ie., it has North pole-South pole dynamics. 
\end{lem} 
\begin{proof}
For the parabolic case, it is an obvious observation. Suppose $f$ is hyperbolic. The only way of not having North pole-South pole dynamics is that both fixed points are parabolic fixed points. But we have $G$-invariant laminations with no leaves connecting the fixed points of $g$ (by the transversality, all but at most one lamination are like that. See Lemma \ref{lem:fixedpointisnotendpoint}). 
 For each of those laminations, there must be a leaf connecting two components of the complement of the fixed points by Lemma \ref{lem:separatinghyperbolicfixedpoints}. They cannot stay unlinked under $f$ if both fixed points are parabolic. 
\end{proof} 

\begin{lem} 
\label{lem:ellipticistorstion}
Let $G$ be a pants-like $\COL_3$ group. Any elliptic element of $G$ is of finite order.
\end{lem} 
\begin{proof} 
 Let $f \in G$ be an elliptic element. If its rotation number is rational, then some power $f^n$ of $f$ must have fixed points. By Corollary \ref{cor:atmost2fixedpoints}, $\Fix_{f^n}$ has either one or two points unless $f^n$ is identity. Suppose first $f^n$ has only one fixed point. Then $f$ must have one fixed point too, contradicting to the assumption. Thus $f^n$ has two fixed points. But since $f$ has no fixed points, both fixed points must be parabolic fixed points, which contradicts to Lemma \ref{lem:parabolichyperbolic}. 
 Hence the only possibility is $f^n = Id$, so $f$ is of finite order. 

 Suppose $f$ has irrational rotation number. It cannot be conjugate to a rigid rotation with irrational angle, since any irrational rotation has no invariant lamination at all as we observed before. 
 So $f$ must be semiconjugate to a irrational rotation, say $R$. We may assume that the action of $f$ on $S^1$ is obtained by Denjoy blow-up for one or several orbits under $R$. Any invariant lamination should be supported by the blown-up orbit. But such a lamination cannot be very full. 
 Hence $f$ cannot have an irrational rotation number. 
\end{proof} 

\begin{thm} 
\label{thm:classification} 
Suppose $G \subset \Homeop(S^1)$ is a pants-like $\COL_3$ group. 

Then each element of $G$ is either a torsion, parabolic, or hyperbolic element. 
\end{thm} 
\begin{proof} 
 This follows from Corollary \ref{cor:atmost2fixedpoints}, Lemma \ref{lem:parabolichyperbolic} and Lemma \ref{lem:ellipticistorstion}. 
\end{proof}

Theorem \ref{thm:classification} provides a classification of elements of pants-like $\COL_3$ groups just like the one for Fuchsian groups. In fact, this is not a coincidence. 
It is hard in general to extend the action to the interior of $\DD$. Instead, we will try to show that pants-like $\COL_3$ groups are convergence groups. Convergence group theorem says that a  group acting faithfully on the circle is a convergence group if and only if it is a Fuchsian group (this theorem was proved for a large class of groups in \cite{Tukia88}, and in full generality in \cite{Gabai91} and \cite{CassonJungreis94}). For the general background for convergence groups, see \cite{Tukia88}.


\section{Pants-like $\COL_3$ Groups are Fuchsian Groups}
\label{sec:convergencegroup}

For general reference, we state the following well-known lemma without proof.
\begin{lem} 
\label{lem:notproperlydiscontinuous}
Let $G$ be a group acting on a space $X$. Let $K$ be a compact subset of $X$ such that $g(K) \cap K \neq \emptyset$ for infinitely many $g_i$. Then there exist a sequence $(x_i)$ in $K$ converging to $x$ and a sequence of the set $\{(g_i)\}$, also call it $(g_i)$ for abusing the notation, such that $g_i(x_i)$ converges to a point $x'$ in $K$.  
\end{lem} 

Let $G$ be a discrete subgroup of $\Homeop(S^1)$. Then $G$ is said to have the \emph{convergence property} if for an arbitrary infinite sequence of distinct elements $(g_i)$ of $G$, there exists two points $a, b \in S^1$ (not necessarily distinct) and a subsequence $(g_{i_j})$ of $(g_i)$ so that $g_{i_j}$ converges to $a$ uniformly on compacts subsets of $S^1 \setminus \{b\}$. If $G$ has the convergence property, then we say $G$ is a \emph{convergence group}. 

Let $T$ be the space of ordered triples of three distinct elements of $S^1$. By Theorem 4.A. in \cite{Tukia88}, a group $G \subset \Homeo(S^1)$ is a convergence group if and only if it acts on $T$ properly discontinuously. If one looks at the proof of this theorem, we do not really use the group operation. Hence we get the following statement from the exactly same proof.
\begin{prop} 
 Let $C$ be a set of homeomorphisms of $S^1$. $C$ has the convergence property if and only if $C$ acts on $T$ properly discontinuously. 
\end{prop}

For pants-like $\COL_3$ group $G$, we can define its limit set in a similar way as for the case of Fuchsian groups.  
Let $\Omega(G)$ be the set of points of $S^1$ where $G$ acts discontinuously, ie., $\Omega(G) = \{ x \in S^1 : \mbox{ there exists a neighborhood } U \mbox{ of } x \mbox{ such that } g(U) \cap U=\emptyset \mbox{ for all but finitely many } g \in G \}$, and call it \emph{domain of discontinuity} of $G$. Let $L(G) = S^1 \setminus \Omega(G)$ and call it \emph{limit set} of $G$. For our conjecture to have a chance to be true, $\Omega(G)$ and $L(G)$ have the same properties as those for Fuchsian groups. 

For the rest of this section, we fix a torsion-free  pants-like $\COL_3$ group $G$ with a pants-like collection $\{\Lambda_1, \Lambda_2, \Lambda_3 \}$ of $G$-invariant laminations.  
Let $F(G)$ be the set of all fixed points of elements of $G$, ie., $F(G) = \cup_{g\in G} \Fix_g$.  
We do not need the following lemma but it shows another similarity of pants-like $\COL_3$ groups with Fuchsian groups. 
\begin{lem}
$\overline{F(G)}$ is either a finite subset of at most 2 points or an infinite set. 
When $\overline{F(G)}$ is infinite, it is either the entire $S^1$ or a perfect nowhere dense subset of $S^1$. 
\end{lem}
\begin{proof}
When all the elements of $G$ share fixed points, then $|\overline{F(G)}| \le 2$ as $|\Fix_g| \le 2$ for all $g \in G$. If that is not the case, say we have $g, h \in G$ whose fixed point sets are distinct. Then for $x \in \Fix_h \setminus \Fix_g$, $g^n(x)$ are all distinct for $n \in \ZZ$ and $g^n(x)$ is a fixed point of $g^n h g^{-n}$, hence $\overline{F(G)}$ is infinite. 
Note that $\overline{F(G)}$ is a closed minimal $G$-invariant subset of $S^1$. An infinite minimal set under the group action on the circle has no isolated points. Thus $\overline{F(G)}$ is a perfect set. 
\end{proof}

Next lemma itself will not be used to prove our main theorem but the proof is important.
\begin{lem} 
\label{lem:limitsetisfixedpoints}
$L(G) = \overline{F(G)}$. 
\end{lem} 
\begin{proof} Since $\overline{F(G)}$ is a minimal closed $G$-invariant subset of $S^1$, it is obvious that $\overline{F(G)} \subset L(G)$.

For the converse, we use laminations. Let $x \in L(G)$. If $x$ is the fixed point of a parabolic element, then we are done. Hence we may assume that it is not the case and 
take $\Lambda_\alpha$ such that $x \notin E_{\Lambda_\alpha}$.
Let $(l_i)$ be a sequence of leaves which forms a rainbow for $x$ (such a sequence exists by Proposition \ref{thm:rainbow}) and let $(I_i)$ be a sequence of open arcs in $S^1$ such that each $I_i$ is the component of the complement of the endpoints of $l_i$ containing $x$. 
Since $G$ acts not discontinuously at $x$, we can choose $g_i \in G$ such that $g_i(I_i)$ intersects $I_i$ nontrivially. This cannot happen arbitrarily, but one must have either $g_i(\overline{I_i}) \subset \overline{I_i}$, $\overline{I_i} \subset g_i(\overline{I_i})$, or $I_i \cup g_i(I_i) = S^1$, since the endpoints of $I_i$ form a leaf. 

In the former two cases, an application of Brower's fixed point theorem implies the existence of a fixed point in $\overline{I_i}$. In the latter one, take any point in $F(G)$: either it 
belongs to $\overline{I_i}$, or its image under $g_i^{-1}$ does
(see Figure \ref{fig:limitsetcases} for possible configurations, and the reason why our situation is restricted to these cases is described in Figure \ref{fig:linkedcase}).
Since $I_i$ shrinks to $x$, this implies that $x$ is a limit point of fixed points of $(g_i)$. 
\end{proof} 

\begin{figure}	
	\centering
	\begin{subfigure}[t]{1.7in}
		\includegraphics[scale=0.4]{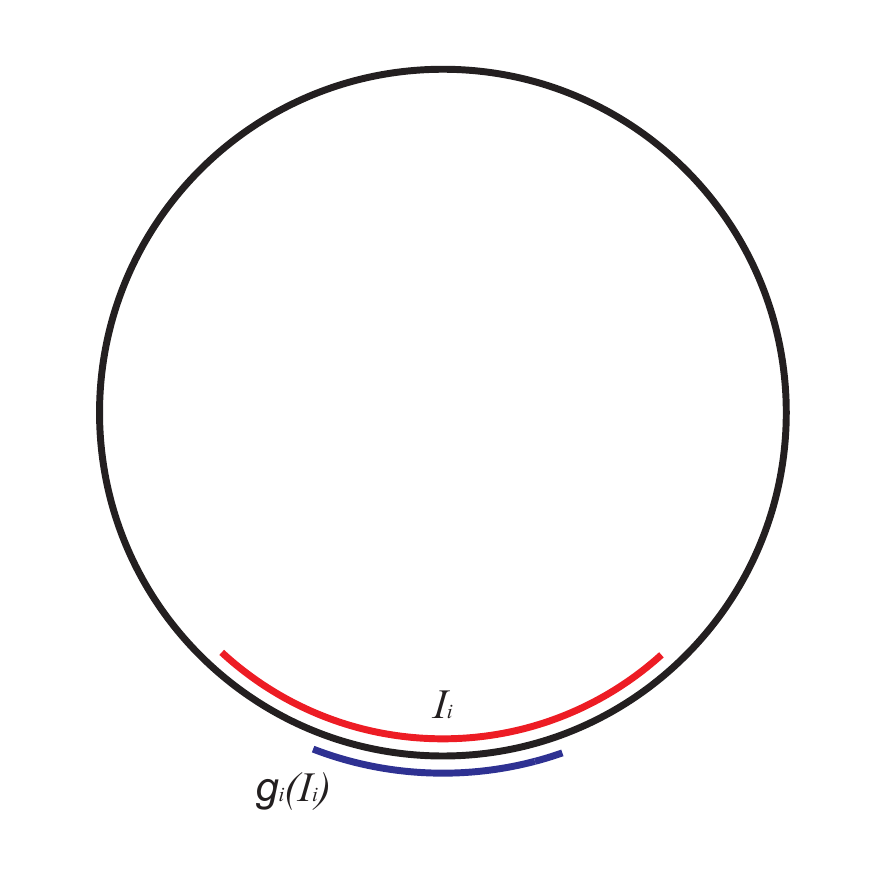}
		\caption{$g_i(I_i)$ is completely contained in $I_i$.}\label{fig:limitset1}		
	\end{subfigure}
	\quad
	\begin{subfigure}[t]{1.7in}
		\includegraphics[scale=0.4]{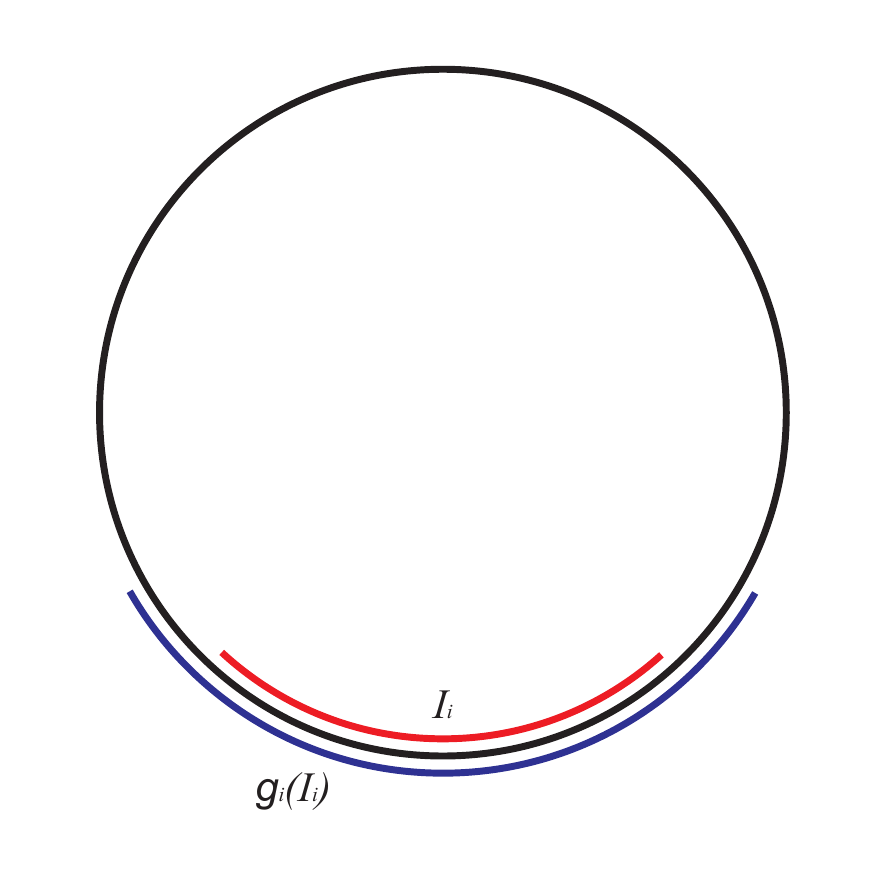}
		\caption{$g_i(I_i)$ completely contains $I_i$.}\label{fig:limitset2}
	\end{subfigure} \\
	\begin{subfigure}[t]{1.7in}
		\includegraphics[scale=0.4]{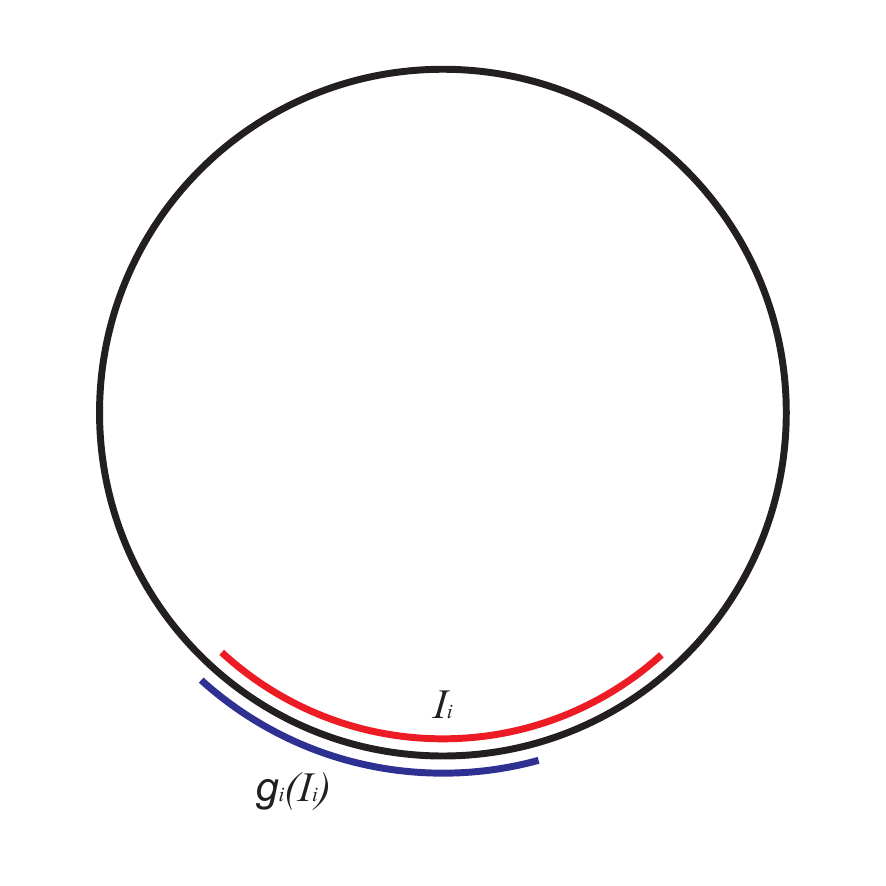}
		\caption{$g_i(I_i)$ is completely contained in $I_i$ but one endpoint of $I_i$ is fixed.}\label{fig:limitset3}
	\end{subfigure}
	\quad
	\begin{subfigure}[t]{1.7in}
		\includegraphics[scale=0.4]{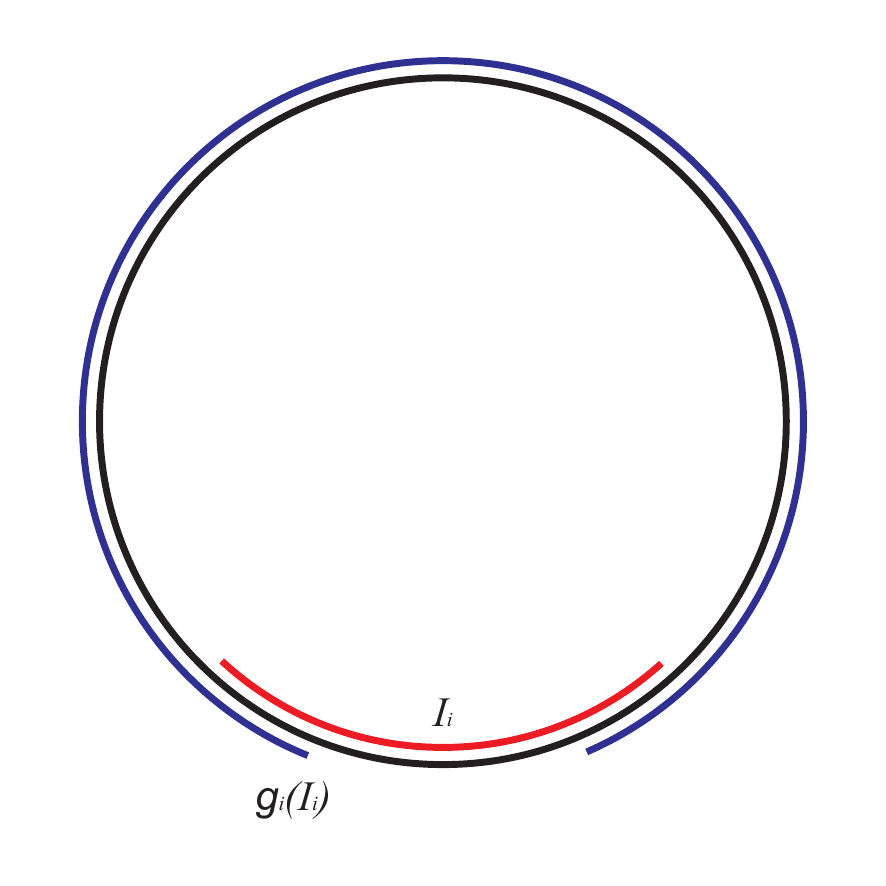}
		\caption{The union of $I_i$ and $g_i(I_i)$ is the whole circle.}\label{fig:limitset4}
	\end{subfigure}
	\caption{This shows the possibilities of the image of $I_i$ under $g_i$. $I_i$ is the red arc (drawn inside the disc) and $g_i(I_i)$ is the blue arc (drawn outside the disc) in each figure.}
	\label{fig:limitsetcases}
\end{figure}

 \begin{figure}[ht]
\begin{center}
\includegraphics[scale=0.4]{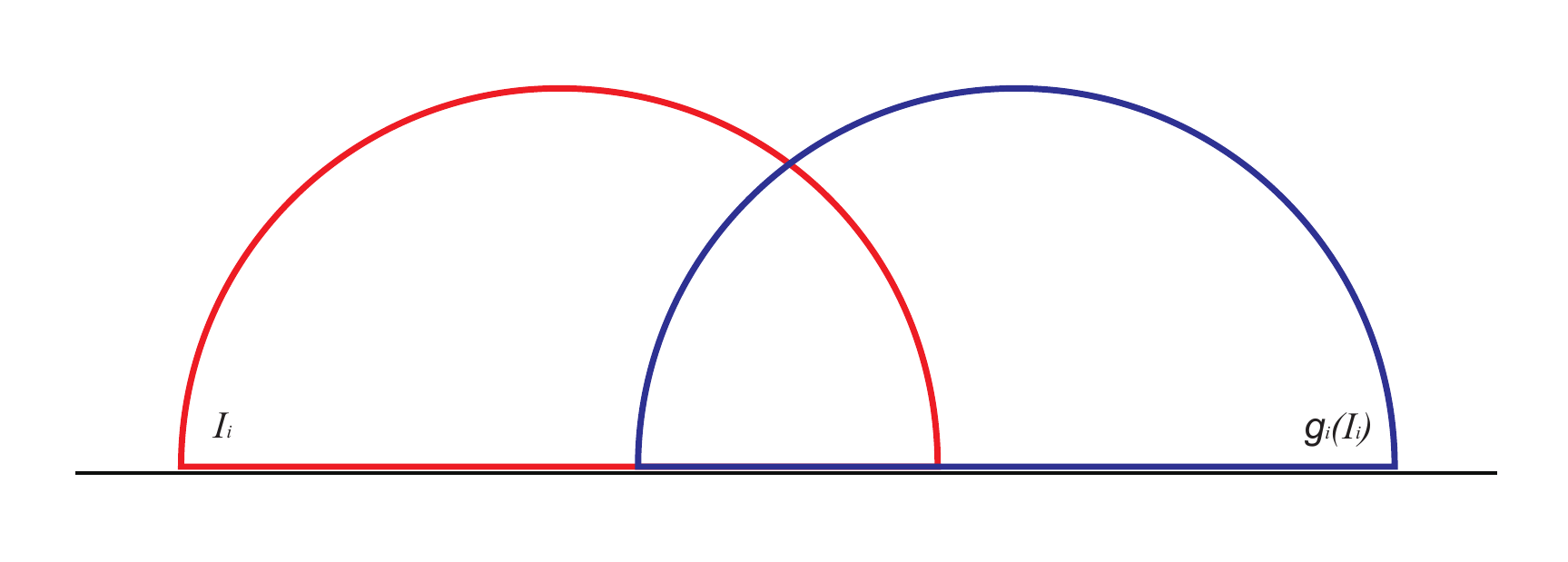}
\end{center}
\caption{ This case is excluded, since the leaf connecting the endpoints of $I_i$ is linked with its image under $g_i$. }
\label{fig:linkedcase}
\end{figure}

\begin{lem} 
\label{lem:discontinuousactionapproximatedbyfixedpoints}
Suppose $(g_i)$ is a sequence of elements of $G$ and $x \in S^1$ such that for any neighborhood $U$ of $x$, $g_i(U)$ intersects $U$ nontrivially for all large enough $i$. Then $x$ is a limit point of the  fixed points of $g_i$ in the sense that there exists a choice of a fixed point $a_i$ of $g_i$ for all $i$ such that the sequence $(a_i)$ converges to $x$. 
\end{lem}
\begin{proof} If we can find $\Lambda_\alpha$ such that $x \notin E_{\Lambda_\alpha}$, then we are done by the proof of Lemma \ref{lem:limitsetisfixedpoints}. Suppose it is not the case, ie., $x$ is the fixed point of a parabolic element $h$. The key point here is to figure out how to construct $I_i$ for $x$ in order to mimic the proof of Lemma \ref{lem:limitsetisfixedpoints}.  Take arbitrary invariant lamination $\Lambda$. There exists a leaf $l$ which has $x$ as an endpoint. Then $h^n(l)$'s form an infinite family of such leaves. For all $i \in \NN$, let $a_i$ be the other endpoint of $h^i(l)$ and $b_i$ be the other endpoint of $h^{-i}(l)$. Let $I_n$ be an open interval containing $x$ with endpoints $a_n, b_n$ for each $n \in \NN$. Now we have a sequence of intervals shrinking to $x$. 
Passing to a subsequence if necessary, we may assume that $g_i(I_{i+1}) \cap I_{i+1} \neq \emptyset$ for all $i$. 

Note that it is possible that neither $g_i(I_{i+1}) \subset I_{i+1}$ nor $I_{i+1} \subset g_i(I_{i+1})$ holds (see Figure \ref{fig:parabolic1}), but one must have either $g_i(I_i) \subset I_i$ or $I_i \subset g_i(I_i)$ to avoid having any linked leaves (see Figure \ref{fig:parabolic2}). Now the same argument shows that $x$ is a limit point of fixed points of $g_i$. 
\end{proof} 

\begin{prop} 
\label{prop:discontinuousactionandthelimitset}
Suppose we have a sequence $(x_i)$ of points in $S^1$ which converges to $x \in S^1$ and a sequence $(g_i)$ of elements of $G$ such that $g_i(x_i)$ converges to $x' \in S^1$. 
Then either $x$ or $x'$ lies in the limit set $L(G)$. Moreover, by passing to a subsequence if necessary, either $x$ is an accumulation point of the fixed points of the sequence $(g_{i+1}^{-1}\circ g_i)$ or $x'$ is an accumulation point of the fixed points of the sequence $(g_i \circ g_{i+1}^{-1})$. 
\end{prop}
\begin{proof} Take $U$ any neighborhood of $x'$. Then for large enough $N$, we have $g_i(x_i) \in U$ for all $i \ge N$. 
We show there is a dichotomy here: either the preimages of $U$ shrink to a point, and we are done quickly, or by passing to a subsequence we can assume all of them to be large. 

Suppose $g_i^{-1}(U)$ does not contain any $x_j$ with $j \neq i$ for each $i \ge N$. 
But $g_i^{-1}(U)$ contains $x_i$ and the sequence $(x_i)$ converges to $x$. Hence $g_i^{-1}(U)$ for $i \ge N$ form a sequence of disjoint open intervals shrinking to $x$, implying that $g_i^{-1}(x')$ converges to $x$. Now let $V$ be a neighborhood of $x$. Replacing $N$ by a larger  number if necessary, $g_i^{-1}(x')$ lies in $V$ for all $i \ge N$. Then $(g_{i+1}^{-1} \circ g_i)(V)$ intersects nontrivially $V$ for all $i \ge N$. 

Now suppose such $U$ does not exist. Take an arbitrary neighborhood $U$ of $x'$. Passing to a subsequence, we may assume that $g_i(x_i) \in U$ for all $i$. By the assumption, for each $i$, there exists $n_i \neq i$ such that $x_{n_i} \in g_i^{-1}(U)$.  Taking a subsequence of $x_i$, we may assume that $(x_i)$ converges to $x$ monotonically. Thus we are allowed to assume that either $n_i  = i+1$ or $n_i = i-1$ for all $i$. In the former case,  $g_{i+1} \circ g_i^{-1} (U)$ intersects $U$ nontrivially for each $i$, and in the latter case, $g_i \circ g_{i+1}^{-1}$ does the same thing (note that $g_i \circ g_{i+1}^{-1}$ and $g_{i+1} \circ g_i^{-1}$ have the same fixed points).

Now the result follows by applying Lemma \ref{lem:discontinuousactionapproximatedbyfixedpoints}. 
\end{proof}

\begin{figure}	
	\centering
	\begin{subfigure}[t]{2.7in}
		\includegraphics[scale=0.44]{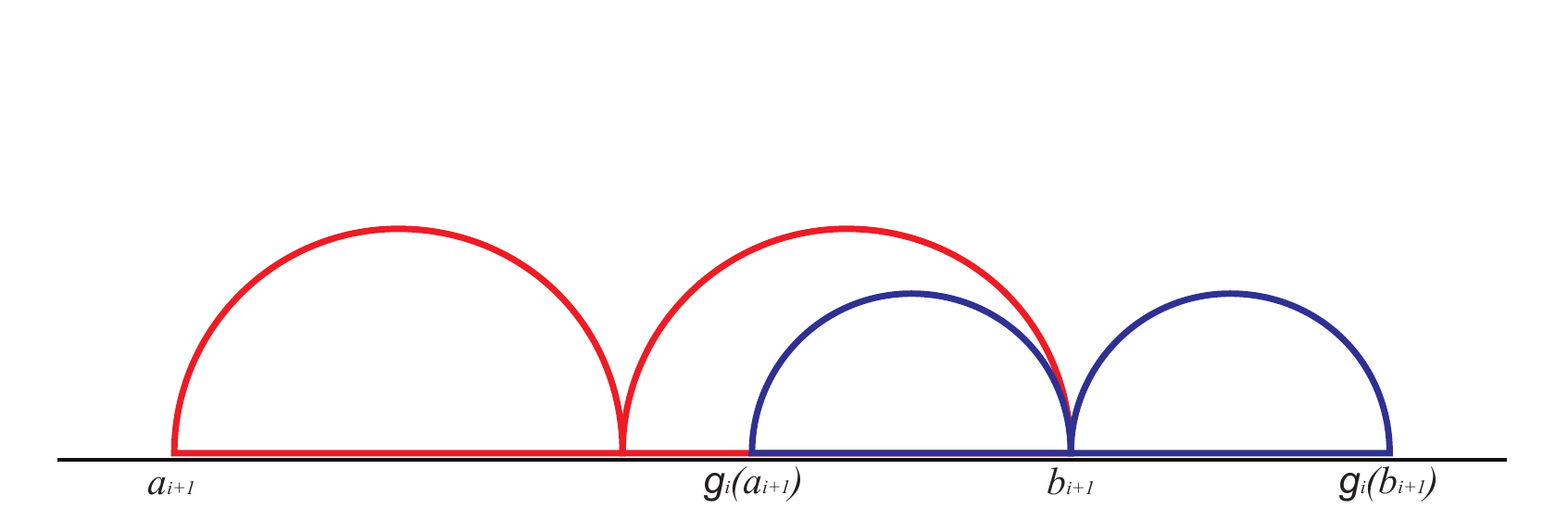}
		\caption{Both $g_i(I_{i+1}) \setminus I_{i+1}$ and $I_{i+1} \setminus g_i(I_{i+1})$ could be non-empty in this case.}\label{fig:parabolic1}		
	\end{subfigure}
	\quad
	\begin{subfigure}[t]{2.7in}
		\includegraphics[scale=0.44]{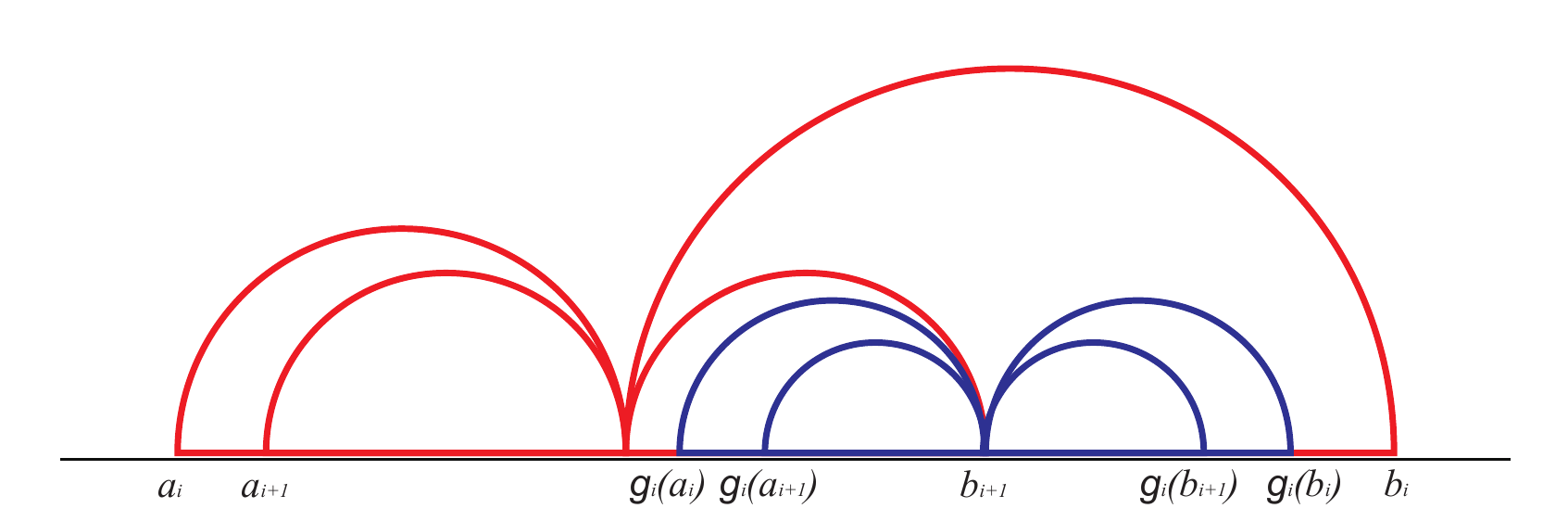}
		\caption{But even in case (a), there is no such a problem for $I_i$ and $g_i(I_i)$. }\label{fig:parabolic2}
	\end{subfigure} 
	\caption{The nested intervals for a cusp point need more care. (a) shows that it might have a problematic intersection, but (b) shows we can take a one-step bigger intervals to avoid that. The endpoints of $I_i$ are marked as $a_i, b_i$ and the endpoints of $I_{i+1}$ are marked as $a_{i+1}, b_{i+1}$. }
	\label{fig:paraboliccases}
\end{figure}

We are ready to prove the main theorem of the paper. 
\begin{thm}
\label{thm:pcol3isconvergence}
Let $G$ be a torsion-free discrete subgroup $\Homeop(S^1)$. If $G$ admits a pants-like collection of very-full invariant laminations $\{\Lambda_1, \Lambda_2, \Lambda_3\}$, then 
$G$ is a Fuchsian group.
\end{thm} 
\begin{proof} 
 By the Convergence Group Theorem, it suffices to prove that $G$ is a convergence group. Suppose not. Then there exists a sequence $(g_i)$ of distinct elements of $G$ such that for any pair of (not necessarily distinct) points $\alpha, \beta$ of $S^1$, no subsequence converges to $\alpha$ uniformly on compact subsets of $S^1 \setminus \{beta\}$. This  implies that this sequence, as a set, acts on $T$ not properly discontinuously. Then we have three sequences $(x_i), (y_i), (z_i)$ converging to $x, y, z$ and a sequence $(h_i)$ of the set $\{g_i\}$ such that $h_i(x_i) \to x', h_i(y_i) \to y', h_i(z_i) \to z'$ where $x, y, z$ are all distinct and $x', y', z'$ are all distinct. Note that the sequence $(h_i)$ could have been taken as a subsequence of $(g_i)$, so let's assume that. 

From Proposition \ref{prop:discontinuousactionandthelimitset}, we can take a subsequence of $(h_i)$ (call it again $(h_i)$ by abusing the notation) such that
either two of $x, y, z$ are accumulation points of the fixed points $h_{i+1}^{-1} \circ h_i$ or two of $x', y', z'$ are accumulation points of fixed points of $h_i \circ h_{i+1}^{-1}$.
Without loss of generality, suppose $x', y'$ are accumulation points of fixed points of $h_i \circ h_{i+1}^{-1}$.

We would like to pass to subsequences so that the fixed points of the sequence $h_{i+1}^{-1} \circ h_i$ (or $h_i \circ h_{i+1}^{-1}$) have at most two accumulation points. But this cannot be done directly, since a subsequence of $(h_{i+1}^{-1} \circ h_i)$ is not from a subsequence of $(h_i)$ in general. Instead, we proceed as follows. 

Take a subsequence $(h_{i_k +1}^{-1} \circ h_{i_k})$ of $(h_{i+1}^{-1} \circ h_i)$ such that there are at most two points where the fixed points of $(h_{i_k +1}^{-1} \circ h_{i_k})$
accumulate (Such a subsequence exists due to Corollary \ref{cor:atmost2fixedpoints} and the compactness of $S^1$). Similarly, let $(h_{i_{k_j}} \circ h_{i_{k_j}+1}^{-1})$ be a subsequence of $(h_i \circ h_{i+1}^{-1})$ such that there are at most two points where the fixed points of 
$(h_{i_{k_j}} \circ h_{i_{k_j}+1}^{-1})$ accumulate. 

Since $x', y'$ are accumulation points of fixed points of $h_i \circ h_{i+1}^{-1}$, they are accumulated points of fixed points of $(h_{i_{k_j}} \circ h_{i_{k_j}+1}^{-1})$. 
But the fixed points of $(h_{i_{k_j}} \circ h_{i_{k_j}+1}^{-1})$ have at most two accumulation points and $x', y', z'$ are three distinct points, so that $z'$ cannot be an accumulation point of fixed points of $(h_{i_{k_j}} \circ h_{i_{k_j}+1}^{-1})$.
This also implies that $z'$ is not an accumulation point of fixed points of $(h_i \circ h_{i+1}^{-1})$. By our choice of $(h_i)$, this implies that $z$ must be an accumulation fixed points of $(h_{i+1}^{-1} \circ h_i)$ (so it is an accumulation point of fixed points of $(h_{i_k +1}^{-1} \circ h_{i_k})$). 

Let $(a_i)$ be the sequence of fixed points of $(h_{i+1}^{-1} \circ h_i)$ which converges to $z$. 
Now we consider a subsequence such that $(h_{i_{k_{j_w}}}(a_{i_{k_{j_w}}})$ converges to a point, say $a$. 
Note that for each $i$, we have $(h_i \circ h_{i+1}^{-1})(h_i(a_i)) = h_i(a_i)$. Hence, $(h_{i_{k_{j_w}}}(a_{i_{k_{j_w}}})$ are fixed points of a subsequence of $(h_{i_{k_j}} \circ h_{i_{k_j}+1}^{-1})$, so that $a$ must be $x'$ or $y'$. Without loss of generality, let's assume that $a= x'$. 

Then we have the followings; 
\begin{itemize}
\item[(1)] $(a_{i_{k_{j_w}}})$ converges to $z$, since $a_i$ converges to $z$. 
\item[(2)] $(h_{i_{k_{j_w}}}(a_{i_{k_{j_w}}}))$ converges to $x'$. 
\item[(3)] $(h_{i_{k_{j_w}}}(z_{i_{k_{j_w}}}))$ converges to $z'$, since $h_i(z_i)$ converges to $z'$. 
\item[(4)] $(h_{i_{k_{j_w}}}(x_{i_{k_{j_w}}}))$ converges to $x'$, since $h_i(x_i)$ converges to $x'$. 
\end{itemize} 

Now for notational simplicity, we drop the subscripts ${i_{k_{j_w}}}$ and simply denote them as $i$ (we are passing to subsequences here). 

We want to have a nested strictly decreasing sequence of intervals for each of $x', z$. Suppose for now that none of them is a cusp point. By the pants-like property, we can take $\Lambda_{\alpha}$ so that no leaf ends in $\{x', z\}$. In this case, we have a rainbow for each of $x', z$ and take intervals as in the proof of Lemma \ref{lem:limitsetisfixedpoints}. For $p \in \{x', z\}$, let $(I^p_i)$ be the sequence of nested decreasing intervals containing $p$. 
By taking subsequences, we may assume that for each $i$, we have $a_i, z_i \in I^z_i$ but $x_i \notin I^z_i$, and $h_i(x_i), h_i(a_i) \in I^{x'}_i$. Then, in particular, $h_i(I^z_i)$ intersects $I^{x'}_i$ non-trivially. But since $h_i$ is a homeomorphism and $h_i(x_i) \in I^{x'}_i$, it is impossible to have $h_i(I^z_i) \supset I^{x'}_i$. Hence there are two possibilities: either $h_i(I^z_i) \subset I^{x'}_i$ or $I^z_i$ is expanded by $h_i$ so that $h_i(I^z_i) \cup I^{x'}_i = S^1$. If the latter happens infinitely often, we can take a subsequence which has only the latter case for all $i$. Then $S^1 \setminus (I^z_i \cup I^{x'}_n)$ is mapped completely into $I^{x'}_n$ by $h_i$. This shows that the sequence $h_i$ has the convergence property with the two points $z, x'$, a contradiction to our assumption. Hence we may assume that this does not happen, ie., $h_i(I^z_i)$ are completely contained in $I^{x'}_n$. But we know that $z_i \in I^z_i$ and $h_i(z_i) \to z' \neq x'$. This is a contradiction.

If some of them are cusp points, we take intervals as in the proof of Lemma \ref{lem:discontinuousactionapproximatedbyfixedpoints}. As we saw, one needs to be slightly more careful to choose $(I^{x'}_i), (I^z_i)$ so that the case $h_i(I^{x'}_i) \not\subset I^z_i, h_i(I^{x'}_i) \not\supset I^z_i$ and $h_i(I^{x'}_i) \cup I^z_i \neq S^1$ does not happen; one can avoid this as we did in the proof of Lemma \ref{lem:discontinuousactionapproximatedbyfixedpoints} (recall the Figure \ref{fig:paraboliccases}). Then the same argument goes through.

Hence the set $\{g_i\}$ must act properly discontinuously on $T$, a contradiction to our assumption. Now the result follows. 
\end{proof}

\begin{rmk} 
\label{rmk:pantslikeproperty}
In the proof of Theorem \ref{thm:pcol3isconvergence}, the consequence of the pants-like property which we needed is that for arbitrary pair of points $p, q \in S^1$ which are not fixed by some parabolic elements, there exists an invariant lamination so that none of $p, q$ is an endpoint of the leaf of that lamination. 
\end{rmk}

\begin{cor}[Main Theorem] 
\label{cor:pcol3ifffuchsian}
Let $G$ be a torsion-free discrete subgroup of $\Homeop(S^1)$. Then $G$ is a pants-like $\COL_3$ group if and only if $G$ is a Fuchsian group whose quotient is not the thrice-punctured sphere. 
\end{cor} 
\begin{proof} 
This is a direct consequence of Theorem \ref{thm:fuchsianispcol3} and Theorem \ref{thm:pcol3isconvergence}.
\end{proof} 

\begin{cor} 
 Let $G$ be a torsion-free discrete subgroup of $\Homeop(S^1)$. Then $G$ admits pairwise strongly transverse three very-full invariant laminations if and only if $G$ is a Fuchsian group whose quotient has no cusps. 
\end{cor} 
\begin{proof} Replacing the pants-like property by the pairwise strong transversality is equivalent to saying that there is no parabolic elements. Hence, this is an immediate corollary of Main Theorem. 
\end{proof} 

\begin{cor}
Let $G$ be a torsion-free discrete pants-like $\COL_3$ group. Then $G$-action on $S^1$ is minimal if and only if $G$ is a pants-decomposable surface group.
\end{cor} 
\begin{proof}
One direction is clear from the observation that the fundamental group of a pants-decomposable surface acts minimally on $\partial_\infty \HH^2$. 
Suppose $G$ is a pants-like $\COL_3$ group. By Theorem \ref{thm:pcol3isconvergence}, $G$ is a Fuchsian group. Let $S$ be the quotient surface $\HH^2 / G$. Note that $S$ is not the thrice-punctured sphere, since it has infinitely many transverse laminations. If $S$ is not pants-decomposable, then there still exists a multi-curve which decomposes $S$ into pairs of pants, half-annuli, and half-planes (Theorem 3.6.2. of \cite{Hubb}). Thus any fundamental domain of $G$-action on $\overline{\HH^2}$ contains some open arcs in $S^1 = \partial_\infty \HH^2$. Let $I$ be a proper closed sub-arc of such an open arc. Since it is taken as a subset of a fundamental domain, the orbit closure of $I$ is a closed invariant subset of $S^1$ which has non-empty interior and is not the whole $S^1$. This contradicts to the minimality of $G$-action. 
\end{proof}

\begin{cor} Let $M$ be a oriented hyperbolic 3-manifold whose fundamental group is finitely generated. If $\pi_1(M)$ admits a pants-like $\COL_3$-representation into $\Homeop(S^1)$, then $M$ is a homeomorphic to $S \times \RR$ for some surface $S$. If we further assume that $M$ has no cusps and is geometrically finite, then $M$ is either quasi-Fuchsian or Schottky. 
\end{cor}
\begin{proof} The existence of a pants-like $\COL_3$-representation into $\Homeop(S^1)$ implies that $\pi_1(M)$ is isomorphic to $\pi_1(S)$ for a hyperbolic surface $S$. Now it is a consequence of the Tameness theorem (independently proved by Agol \cite{AgolTame} and Calegari-Gabai \cite{CalegariGabaiTame}). 
\end{proof}

\begin{rmk} There is an analogy between the cardinality of the set of ends of groups and the cardinality of the paths-like collection of laminations that subgroups of $\Homeop(S^1)$ can have. In Theorem \ref{thm:fuchsianispcol3}, one can work harder to show that Fuchsian groups are in fact pants-like $\COL_\infty$ groups. The result of Section 3 says there are pants-like $\COL_2$ groups which are not pants-like $\COL_3$ groups (we will see the distinction in more detail in the next section). Hence, any torsion-free discrete subgroup of $\Homeop(S^1)$ is a pants-like $\COL_n$ group where $n$ is either $0, 1, 2,$ or infinity, while the cardinality of the set of ends of a group has the same possibilities.  \end{rmk}

\begin{rmk} In Theorem \ref{thm:pcol3isconvergence}, it is easy to see that the torsion-free assumption is not necessary. We conjecture that main theorem (Corollary \ref{cor:pcol3ifffuchsian}) could be stated without the torsion-free assumption. To show that, one needs to construct pants-like collection of three very-full laminations on hyperbolic orbifolds. It is not too clear how to do so with simple geodesics. 
\end{rmk}

\section{Pants-like $\COL_2$ Groups and Some Conjectures}
\label{sec:pcol2}
 We saw that being torsion-free discrete pants-like $\COL_3$ is equivalent to being Fuchsian. In this section, we will try to see what is still true if we have one less lamination. For the rest of this section, we fix a pants-like $\COL_2$ group $G$ with a pants-like collection $\{\Lambda_1, \Lambda_2\}$ of $G$-invariant laminations. For the sake of simplicity, we also assume that $|\Fix_g| < \infty$ for each $g \in G$.

 \begin{prop}
 \label{prop:nonparabolicofpcol2}
 Let $g$ be a non-parabolic element of $G$. Then $g$ have no parabolic fixed point. Hence either $g$ is elliptic or $g$ has even number of fixed points which alternate between attracting fixed points and repelling fixed points along $S^1$. 
 \end{prop}
 \begin{proof}
 Suppose $\Fix_g \neq \emptyset$. Let $I$ be a connected component of $S^1 \setminus \Fix_g$ with endpoints $a$ and $b$. In the previous section, we saw that for each $i$, either $a \in E_{\Lambda_i}$ or $b \in E_{\Lambda_i}$. We also know that none of $a$ and $b$ can be the fixed point of a parabolic element (see the second half of the proof of Corollary \ref{cor:atmost2fixedpoints}). Hence the pants-like property implies that there is no $i$ such that both $a$ and $b$ are in $E_{\Lambda_i}$. In particular, this implies that for each $p \in \Fix_g$, there exists $i \in \{1, 2\}$ so that $p$ is not in $E_{\Lambda_i}$. But this implies that there is a rainbow in $\Lambda_i$ at $p$. But a parabolic fixed point cannot have a rainbow. This proves the claim.  
 \end{proof} 
 
 \begin{cor} 
 Each elliptic element of $G$ is either of finite order or pseudo-Anosov-like. 
 \end{cor}
 \begin{proof} 
 This is a consequence of Lemma \ref{lem:ellipticistorstion} and Proposition \ref{prop:nonparabolicofpcol2}. 
 \end{proof} 
 
 We have proved the following. 
 \begin{thm}[Classification of Elements of Pants-like $\COL_2$ Groups]
 Let $G$ be as defined at the beginning of the section. The elements of $G$ are either torsion, parabolic, hyperbolic or pseudo-Anosov-like. 
 \end{thm} 
 
 \begin{conj}
Suppose $G$ is torsion-free discrete and $|\Fix_g| \le 2$ for each $g \in G$. Then $G$ is Fuchsian. 
\end{conj}

 For each pseudo-Anosov-like element $g$ of $G$, let $n = n(g)$ be the smallest positive number such that $g^n$ have fixed points. The boundary leaves of the convex hull of the attracting fixed points form a ideal polygon, we call it the \emph{attracting polygon} of $g$. The \emph{repelling polygon} of $g$ is defined similarly. 
 
 \begin{thm}
 \label{thm:polygonsinlam}
Let $G, \Lambda_1, \Lambda_2$ be as defined at the beginning of the section.  Suppose that there exists $g \in G$ which has more than two fixed points (so there are at least 4 fixed points).  Then each $\Lambda_i$ contains either the attracting polygon of $g$ or the repelling polygon of $g$. \end{thm}
 \begin{proof}
 Say $\Fix_g = \{p_1, \ldots, p_n\}$ such that if we walk from $p_i$ along $S^1$ counterclockwise, then the first element of $\Fix_g$ we meet is $p_{i+1}$ (indexes are modulo $n$). Suppose $p_1 \in E_{\Lambda_1}$. Then by the argument in the proof of Proposition \ref{prop:nonparabolicofpcol2}, both $p_2$ and $p_n$ are not in $E_{\Lambda_1}$. If we apply this consecutively, one can easily see that $p_i \in E_{\Lambda_1}$ if and only if $i$ is odd. 
 
 Let $j$ be any even number. Since $p_j$ is not in $E_{\Lambda_1}$, there exists a rainbow at $j$. In particular, there exists a leaf $l$ in $\Lambda_1$ so that one end of $l$ lies between $p_j$ and $p_{j+1}$, and the other end lies between $p_j$ and $p_{j-1}$. Hence either $g^n(l)$ or $g^{-n}(l)$ converges to the leaf $(p_{j-1}, p_{j+1})$ as $n$ increases. So, the leaf $(p_{j-1}, p_{j+1})$ should be contained in $\Lambda_1$. Since $j$ was an arbitrary even number, this shows that $\Lambda_1$ contains the boundary leaves of the convex hull of the fixed points of $g$ with odd indices. 
 Similarly, one can see that $\Lambda_2$ must contain the boundary leaves of the convex hull of the fixed points of $g$ with even indices. Since the fixed points of $g$ alternate between attracting and repelling fixed points along $S^1$, the results follows. 
 \end{proof} 
 
 This shows that not only the pseudo-Anosov-like elements resemble the dynamics of pseudo-Anosov homeomorphisms but also their invariant laminations are like stable and unstable laminations of pseudo-Anosov homeomorphisms. 
 
 We introduce a following useful theorem of Moore \cite{MooreSphere} and an application in our context. 
\begin{thm}[Moore] 
 Let $\mathcal{G}$ be an upper semicontinuous decomposition of $S^2$ such that each element of $\mathcal{G}$ is compact and nonseparating. Then $S^2/\mathcal{G}$ is homeomorphic to $S^2$. 
\end{thm}
 
 A decomposition of a Hausdorff space $X$ is \emph{upper semicontinuous} if and only if the set of pairs $(x,y)$ for which $x$ and $y$ belong to the same decomposition element is closed in $X \times X$. A lamination $\Lambda$ of $S^1$ is called \emph{loose} if no point on $S^1$ is an endpoint of two leaves of $\Lambda$ which are not edges of a sinlge gap of $\Lambda$. 
 
\begin{thm} 
\label{thm:Actionon2sphere}
Let $G, \Lambda_1, \Lambda_2$ be as defined at the beginning of the section. We further assume that $G$ is torsion free and each $\Lambda_i$ is loose. Then $G$ acts on $S^2$ by homeomorphisms such that $|\Fix_g(S^2) := \{p \in S^2 : g(p) = p\}| \le 2$ for each $g \in G$. 
\end{thm}
\begin{proof} Let $D_1$ and $D_2$ are disks glued along their boundaries, and we consider this boundary as the circle where $G$ acts. So we get a 2-sphere, call it $S_1$, such that $G$ acts on its equatorial circle. Put $\Lambda_i$ on $D_i$ for each $i = 1, 2$. One can first define a relation on $S_1$ so that two points are related if they are on the same leaf or the same complementary region of $\Lambda_i$ for some $i$. Let $\sim$ be the closed equivalence relation generated by the relation we just defined. 

  It is fairly straightforward to see that $\sim$ satisfies the condition of Moore's theorem from the looseness. Looseness, in particular, implies that each equivalence class of $\sim$ has at most finitely many points in $S^1$. 
  
  
  
  This concludes that $S_2 := S_1/\sim$ is homeomorphic to a 2-sphere, and let $p : S_1 \to S_2$ be the corresponding quotient map. Clearly, $p$ is surjective even after restricted to the equatorial circle, call it $p$ again. Now we have a quotient map $p : S^1 \to S_2 = S^2$, hence $G$ has an induced action on $S^2$ by homeomorphisms. 
 Note that $|\Fix_g(S^1)| \ge |\Fix_g(S^2)|$ for each $g \in G$. But we know that if $g \in G$ has more than two fixed points on $S^1$, its attracting fixed points are mapped to a single point by $p$ by Theorem \ref{thm:polygonsinlam}. Similarly, the repelling fixed points are mapped to a single point too. Hence, $g$ can have at most two fixed points in any case. 
\end{proof}
 
 The assumption that $G$ does not have parabolic elements seems unnecessary, but it is probably much trickier to prove that each equivalence class of $\sim$ is non-separating under the existence of parabolic elements. It is also not so clear if the action on $S^2$ we obtained in the above theorem is always a convergence group action. 

From what we have seen, it is conceivable that $G$ contains a subgroup of the form $H \rtimes \ZZ$ where $H$ is a pants-like $\COL_3$ group and $\ZZ$ is generated by a pseudo-Anosov like element (unless $G$ itself is a pants-like $\COL_3$ group). Maybe one can hope the following conjecture to be true (possibly modulo Cannon's conjecture \cite{CannonConjecture}). 

\begin{conj} 
\label{conj:3manifoldgroupversionconjecture}
Let $G$ be a finitely generated torsion-free discrete subgroup $\Homeop(S^1)$. Then $G$ is virtually a pants-like $\COL_2$ group with loose laminations if and only if $G$ is virtually a hyperbolic 3-manifold group. 
\end{conj}

If $G$ is a hyperbolic 3-manifold group, then Agol's Virtual Fibering theorem in \cite{AgolVF} says that $G$ has a subgroup of finite index which fibers over the circle. Hence the result of Section 3 implies that such a subgroup is $\COL_2$. The laminations we have are stable and unstable laminations of a pseudo-Anosov map of a hyperbolic surface, hence they form a pants-like collection of two very-full laminations. This proves one direction of the conjecture. 
 Cannon's conjecture says that if a word-hyperbolic group with ideal boundary homeomorphic to $S^2$ acts on its boundary faithfully, then the group is a Kleinian group. To prove the converse of Conjecture \ref{conj:3manifoldgroupversionconjecture} requires one to show that a pants-like $\COL_2$ group has a subgroup of finite index which is word-hyperbolic and acts on $S^2$ as a convergence group. 




\section{Future Directions} 
\label{sec:fucturedirection}
We have seen that having two or three very-full laminations restricts the dynamics of the group action quite effectively. One can still study what we can conclude about the group when we have mere dense laminations (not necessarily very-full). In any case, the most interesting question is about the difference between having two laminations and three laminations. Thurston conjectured that tautly foliated 3-manifold groups are strictly $\COL_2$ (we know that they are $\COL_2$). Hence, one can ask following questions:

 
 \begin{ques} 
 What algebraic properties of a group $G$ we could deduce from the assumption that $G$ is strictly $\COL_2$, ie., $\COL_2$ but not $\COL_3$? 
 \end{ques} 
 
 \begin{ques}
 Precisely which 3-manifold groups are strictly $\COL_2$? 
 \end{ques} 
 
  More ambitiously, one may ask: 
  
  \begin{ques}
  Can one construct an interesting geometric object like a taut foliation or an essential lamination in a 3-manifold $M$ if we know $\pi_1(M)$ is strictly $\COL_2$? 
  \end{ques}

  It would be also interesting if one can characterize the difference between strictly $\COL_1$ groups and $\COL_2$ group. The example of strictly $\COL_1$ group we constructed suggests that in order to be $\COL_2$, a group should not have too many homeomorphisms with irrational rotation number. It is conceivable that the way the example constructed is essentially the only way to get strictly $\COL_1$ property.

 Another important direction would be to classify all possible $\COL_n$-representations of an abstract group $G$. This is related to the classification of all circular orderings on $G$. The author is preparing a paper about the action of the automorphims of $G$ on the space of all circular orderings which $G$ can admit. For example, the $\Aut(G)$ acts faithfully on the space of circular ordering of $G$ if $G$ is residually torsion-free nilpotent. 
 
We also remark that the virtual fibering theorem of Agol and the universal circle theorem for the fibering case together imply that the following conjecture holds if Cannon's conjecture holds. 
 
 \begin{conj} 
 Let $G$ be a word-hyperbolic group whose ideal boundary is homeomorphic to a 2-sphere, and suppose that $G$ acts faithfully on its boundary. Then $G$ is virtually $\COL_2$. 
 \end{conj}
 
 At the end of the last section, we formulated a conjecture about being virtually a 3-manifold group which fibers over $S^1$. In the view of Vlad Markovic's recent work in \cite{MarkovicCCC}, pants-like $\COL_3$ subgroups of a pants-like $\COL_2$ word-hyperbolic group $G$ are good candidates for quasi-convex codimension-1 subgroups whose limit sets separate pairs of points in the boundary of $G$. 

\bibliographystyle{alpha}
\bibliography{biblio}

\end{document}